\documentclass[oneside,english]{amsart}
\usepackage[T1]{fontenc}
\usepackage[latin9]{inputenc}
\usepackage{color}
\usepackage{babel}
\usepackage{float}
\usepackage{units}
\usepackage{url}
\usepackage{amstext}
\usepackage{amsthm}
\usepackage{amssymb}
\usepackage{graphicx}
\usepackage{xargs}[2008/03/08]
\usepackage[all]{xy}
\usepackage[unicode=true,pdfusetitle,
 bookmarks=true,bookmarksnumbered=false,bookmarksopen=false,
 breaklinks=true,pdfborder={0 0 0},pdfborderstyle={},backref=false,colorlinks=true]
 {hyperref}

\makeatletter

\newcommand{\noun}[1]{\textsc{#1}}

\numberwithin{equation}{section}
\numberwithin{figure}{section}
\theoremstyle{remark}
\newtheorem*{acknowledgement*}{\protect\acknowledgementname}
\theoremstyle{plain}
\newtheorem{thm}{\protect\theoremname}[section]
\theoremstyle{definition}
\newtheorem{defn}[thm]{\protect\definitionname}
\theoremstyle{remark}
\newtheorem{rem}[thm]{\protect\remarkname}
\newcommand\thmsname{\protect\theoremname}
\newcommand\nm@thmtype{theorem}
\theoremstyle{plain}

\newenvironment{namedthm}[1][Undefined Theorem Name]{
  \ifx{#1}{Undefined Theorem Name}\renewcommand\nm@thmtype{theorem*}
  \else\renewcommand\thmsname{#1}\renewcommand\nm@thmtype{namedtheorem}
  \fi
  \begin{\nm@thmtype}}
  {\end{\nm@thmtype}}
\newenvironment{lyxlist}[1]
	{\begin{list}{}
		{\settowidth{\labelwidth}{#1}
		 \setlength{\leftmargin}{\labelwidth}
		 \addtolength{\leftmargin}{\labelsep}
		 }}
	{\end{list}}
\theoremstyle{plain}
\newtheorem{lem}[thm]{\protect\lemmaname}
\theoremstyle{plain}
\newtheorem{prop}[thm]{\protect\propositionname}
\theoremstyle{definition}
\newtheorem*{problem*}{\protect\problemname}
\theoremstyle{plain}
\newtheorem{cor}[thm]{\protect\corollaryname}

\usepackage{kpfonts}
\usepackage{esint}

\makeatother

\providecommand{\acknowledgementname}{Acknowledgement}
\providecommand{\corollaryname}{Corollary}
\providecommand{\definitionname}{Definition}
\providecommand{\lemmaname}{Lemma}
\providecommand{\problemname}{Problem}
\providecommand{\propositionname}{Proposition}
\providecommand{\remarkname}{Remark}
\providecommand{\theoremname}{Theorem}

\begin{document}


\global\long\def\tx#1{\mathrm{#1}}%
\global\long\def\dd#1{\tx d#1}%
\global\long\def\tt#1{\mathtt{#1}}%
\global\long\def\ww#1{\mathbb{#1}}%
\global\long\def\DD#1{\tx D#1}%
\global\long\def\nf#1#2{\nicefrac{#1}{#2}}%
\global\long\def\group#1{{#1}}%

\newcommand{\bigslant}[2]{{\raisebox{.3em}{$#1$}/\raisebox{-.3em}{$#2$}}}


\global\long\def\quot#1#2{\bigslant{#1}{#2}}%

\global\long\def\rr{\mathbb{R}}%
\global\long\def\rbar{\overline{\mathbb{R}}}%
\global\long\def\cc{\mathbb{C}}%
\global\long\def\cbar{\overline{\cc}}%
\global\long\def\zz{\mathbb{Z}}%
\global\long\def\zp{\mathbb{Z}_{\geq0}}%
\newcommandx\zsk[1][usedefault, addprefix=\global, 1=k]{\nicefrac{\zz}{#1\zz}}%
\global\long\def\qq{\mathbb{Q}}%
\global\long\def\qbar{\overline{\qq}}%

\global\long\def\rat#1{\cc\left(#1\right) }%
\global\long\def\pol#1{\cc\left[#1\right] }%
\global\long\def\id{\tx{Id}}%

\global\long\def\GL#1#2{\tx{GL}_{#1}\left(#2\right)}%


\global\long\def\fol#1{\mathcal{F}_{#1}}%
\global\long\def\pp#1{\frac{\partial}{\partial#1}}%
\global\long\def\ppp#1#2{\frac{\partial#1}{\partial#2}}%
\newcommandx\sat[2][usedefault, addprefix=\global, 1=\fol{}]{\tx{Sat}_{#1}\left(#2\right)}%
\global\long\def\lif#1{\mathcal{L}_{#1}}%
\newcommandx\holo[1][usedefault, addprefix=\global, 1=\gamma]{\mathfrak{h}_{#1}}%
\global\long\def\sing#1{\tx{Sing}\left(#1\right)}%
\global\long\def\flow#1#2#3{\Phi_{#1}^{#2}#3}%
\global\long\def\ddd#1#2{\frac{\dd{#1}}{\dd{#2}}}%
\newcommandx\per[1][usedefault, addprefix=\global, 1=R]{\mathfrak{T}_{#1}}%
\global\long\def\sec#1{\mathfrak{S}_{#1}}%


\global\long\def\ii{\tx i}%
\global\long\def\ee{\tx e}%

\global\long\def\re#1{\Re\left(#1\right)}%
\global\long\def\im#1{\Im\left(#1\right)}%

\global\long\def\sgn#1{\mbox{sign}\left(#1\right)}%
\global\long\def\floor#1{\left\lfloor #1\right\rfloor }%
\global\long\def\ceiling#1{\left\lceil #1\right\rceil }%


\global\long\def\germ#1{\cc\left\{  #1\right\}  }%
\global\long\def\frml#1{\cc\left[\left[#1\right]\right] }%


\newcommandx\norm[2][usedefault, addprefix=\global, 1=\bullet]{\left\Vert #1\right\Vert _{#2}}%
\newcommandx\neigh[2][usedefault, addprefix=\global, 1=, 2=0]{\left(\cc^{#1},#2\right)}%


\newcommandx\proj[2][usedefault, addprefix=\global, 1=1, 2=\cc]{\mathbb{P}_{#1}\left(#2\right)}%
\global\long\def\sone{\mathbb{S}^{1}}%


\global\long\def\longinj#1#2{\xymatrix{#1\ar@{^{(}->}[r]  &  #2}
 }%
\global\long\def\longsurj#1#2{\xymatrix{#1\ar@{->>}[r]  &  #2}
 }%
\global\long\def\inj{\hookrightarrow}%
\global\long\def\surj{\twoheadrightarrow}%
\global\long\def\longto{\longrightarrow}%
\global\long\def\longmaps{\longmapsto}%
\global\long\def\cst{\tx{cst}}%
\global\long\def\oo#1{\tx o\left(#1\right)}%
\global\long\def\OO#1{\tx O\left(#1\right)}%

\newcommandx\diff[1][usedefault, addprefix=\global, 1={\cc^{2},0}]{\tx{Diff}\left(#1\right)}%
\newcommandx\holf[1][usedefault, addprefix=\global, 1={\cc^{2},0}]{\tx{Holo}\left(#1\right)}%
\newcommandx\holb[1][usedefault, addprefix=\global, 1={\cc^{2},0}]{\tx{Holo_{b}}\left(#1\right)}%

\global\long\def\polg#1#2{\pol{#1}_{\leq2k}\left\{  #2\right\}  _{>0} }%
\global\long\def\eqtag{(\star)}%
\global\long\def\xx#1{#1}%
\global\long\def\xp{\xx +}%
\global\long\def\xm{\xx -}%
\global\long\def\zi{\xp\xm}%
\global\long\def\iz{\xm\xp}%
\newcommandx\obj[3][usedefault, addprefix=\global, 1=j, 2=\sharp]{\,^{#1}#3^{#2}}%
\newcommandx\sect[2][usedefault, addprefix=\global, 1=, 2=]{\obj[#1][#2]V}%
\newcommandx\ssect[2][usedefault, addprefix=\global, 1=, 2=]{\obj[#1][#2]{\mathcal{V}}}%
\newcommandx\hh[3][usedefault, addprefix=\global, 1=, 2=, 3=N]{\obj[#1][#2]{_{#3}\mathcal{H}}}%

\title{Analytic normal forms for planar resonant saddle vector fields}
\author{Loïc TEYSSIER}
\date{November 2022}
\address{Laboratoire I.R.M.A.\\
Université de Strasbourg}
\email{\url{teyssier@math.unistra.fr}}
\begin{abstract}
We give essentially unique ``normal forms'' for germs of a holomorphic
vector field of the complex plane in the neighborhood of an isolated
singularity which is a $p:q$ resonant-saddle. Hence each vector field
of that type is conjugate, by a germ of a biholomorphic map at the
singularity, to a preferred element of an explicit family of vector
fields. These model vector fields are polynomial in the resonant monomial.

This work is a followup of a similar result obtained for parabolic
diffeomorphisms which are tangent to the identity, and solves the
long standing problem of finding explicit local analytic models for
resonant saddle vector fields.
\end{abstract}

\maketitle
\begin{acknowledgement*}
The author is partially supported by the bilateral Hubert-Curien Cogito
grant 2021-22.
\end{acknowledgement*}

\section{Introduction and statement of the results}

The general question of finding a simpler form, or ultimately ``the''
simplest form, of a dynamical system through changes preserving its
qualitative properties is central in the theory. A simpler form often
means a better understanding of the behavior of the system, or of
its analytic properties. This article is concerned with finding simple
models for holomorphic dynamical systems given by the flow of a $p:q$
resonant-saddle vector field (eigenratio $-\nf pq$) near the origin
of $\ww C^{2}$. (Precise definitions are given later in the introduction.)
We use intensively the appellation \emph{normal form} for vector fields
brought into these forms. Although the latter do not satisfy algebraic
properties usually required in normal form theory, its usage is nonetheless
spreading to refer to preferred forms which are \emph{essentially
unique} (say, up to the action of a finite-dimensional space).

It is possible to attach to a vector field $Z:=A\pp x+B\pp y$ two
dynamical systems: the one induced by the flow, and the underlying
foliation. In the former setting the objects of study are the trajectories
of $Z$ and their natural parameterization by the time, \emph{i.e.}
maximally-continued multivalued solutions of the autonomous differential
system with complex time
\begin{eqnarray}
\begin{cases}
\dot{x}\left(t\right) & =A\left(x\left(t\right),y\left(t\right)\right)\\
\dot{y}\left(t\right) & =B\left(x\left(t\right),y\left(t\right)\right)
\end{cases} &  & ,\label{eq:flow_system}
\end{eqnarray}
while in the latter only their images are of interest: the leaves
of the foliation $\mathcal{F}_{Z}$ are the integral curves of $Z$
regardless of how they are parameterized. Save for vertical leaves,
they correspond to the graphs of solutions of the nonautonomous ordinary
differential equation (since $\frac{\dd y}{\dd x}=\frac{\dot{y}}{\dot{x}}$)
\begin{eqnarray*}
A\left(x,y\left(x\right)\right)y'\left(x\right) & = & B\left(x,y\left(x\right)\right)\,.
\end{eqnarray*}
 Therefore two vector fields induce the same foliation when they differ
by the multiplication with a nonvanishing function (a holomorphic
unit). 

\subsection{A brief survey of the normal form problem for planar resonant singularities}

Being given a (germ of a) holomorphic vector field, we seek to simplify
its components by use of local analytic changes of coordinates. At
first one would try and simplify the vector field using formal power
series, and for planar vector field this process leads to polynomial
formal normal forms~\cite{Brjuno,Dulac}. Yet this formal approach
does not always preserve the dynamics, as is particularly the case
in the presence of resonances where divergence of the formal normalization
is the rule. 

Analyzing the divergence of these formal transforms provides many
an information about the dynamics or the integrability (in the sense
of Liouville) of the system. The theory of summability was used successfully
by J.~\noun{Martinet} and J.-P.~\noun{Ramis~}\cite{MaRa-SN,MaRa-Res}
to perform this task for saddle-node and resonant vector fields, ultimately
yielding a complete set of functional invariants that classifies the
foliation (called here the \emph{orbital modulus}). However, their
construction did not yield normal forms except in very exceptional
(integrable) cases. Some years later the complete modulus of resonant
and saddle-node vector fields (eigenratio $0$) was described in~\cite{VoroGrintch,VoroMesh,Tey-SN}
by appending to the orbital modulus another functional invariant,
called here the \emph{temporal modulus} and accounting for the multiplicative
units that give rise to different vector fields while keeping the
same foliation. Still no normal form was proposed. Analytic normal
forms were announced in~\cite{BrjunoAna} but no proof was subsequently
published. 

Building on an earlier work of P.~\noun{Elizarov}~\cite{Eliza},
a prenorrmal (nonunique) form is presented in~\cite{Tey-ExSN} that
allows to decide in some cases whether two vector fields (or foliations)
are not conjugate. At about the same time F.~\noun{Loray~}\cite{Loray}
performed a cleverly simple geometric construction that yielded normal
forms for codimension-1 saddle-node foliations, but only in the nongeneric
case where half the orbital modulus is nontrivial (\emph{convergent
saddle-nodes}, admitting two analytic separatrices through the singularity).
Loray's normal forms generalized the ones stated by J.~\noun{Écalle}
in~\cite{Ecal} (see the paper by D.~\noun{Sauzin}~\cite{Sauzin}
for precise statements and proofs regarding the resurgent approach
to saddle-node classification). In Écalle's terminology, convergent
saddle-nodes are called \emph{unilateral}, and save for that case
no general normal forms were given. Later, Écalle refined his theory
and techniques to propose a way of building a preferred representative
in the analytic class of a given resonant vector field~\cite{EcalReal},
the \emph{canonical-spherical synthesis}. Although uniqueness is reached,
this approach does not provide an explicit family of vector fields
that can be written down.

In joint work with R.~\noun{Schäfke,} an altogether different approach
was used in~\cite{SchaTey} to recover Loray normal forms, while
extending them from foliations to vector fields and generalizing them
to higher-codimension saddle-nodes. Based on a holomorphic fixed-point
method, it was later reused with C.~\noun{Rousseau} to encompass
the parametric case in~\cite{RT2}, in order to provide normal forms
for convergent saddle-node bifurcations, while at the same time Loray's
construction was ported to parametric families. Yet it was not possible
to drop the nongeneric assumption regarding the ``unilaterality''
of the orbital modulus. The reason behind this difficulty is explained
later, let us for now simply state that the remedy lies in introducing
a parameter, playing the same role as Écalle's \emph{twist} in the
twisted resurgent monomials that serve as building blocks for the
canonical-spherical synthesis. The trick was already used in~\cite{TeySphere}
to provide normal forms for general germs of a parabolic line biholomorphism
which is tangent to the identity. The latter paper was written with
the clear aim of overcoming the problem and porting the technique
to general planar resonant vector fields, and I encountered Écalle's
work on twisted monomials during the final stage of its redaction. 

The present paper is a blend of holomorphic fixed-point and twist
parameter, and it achieves the task of finding a general explicit
normal form family for resonant planar saddle vector fields.

\subsection{Statement of the main result}

Consider a planar holomorphic vector field $Z$ near some isolated
stationary point (or \textbf{singularity}), which we conveniently
locate at $\left(0,0\right)$ so that $Z\left(0,0\right)=0$. Its
Jacobian matrix at that point admits two eigenvalues, $\lambda_{1}$
and $\lambda_{2}$, at least one of which we assume nonzero (the origin
is a nondegenerate singular point of $Z$), say $\lambda_{2}\neq0$.
The \textbf{eigenratio} $\lambda:=\frac{\lambda_{1}}{\lambda_{2}}$
encodes an important part of the dynamics. It is well known for instance
that if $\lambda\notin\rr$, then there exists a local biholomorphic
mapping $\Psi~:~\neigh[2]\to\neigh[2]$, a property that we write
$\Psi\in\diff$, such that the pullback vector field
\begin{align*}
\Psi^{*}Z & :=\tx D\Psi^{-1}\left(Z\circ\Psi\right)
\end{align*}
is linear\emph{:} $\Psi^{*}Z\left(x,y\right)=\lambda_{1}x\pp x+\lambda_{2}y\pp y$.
We say that $Z$ is \textbf{analytically conjugate} to its linear
part (or analytically linearizable).

Of course when $\lambda\in\rr$ it may happen that $Z$ is not analytically
linearizable but if $\lambda>0$, then $Z$ is analytically conjugate
to a polynomial vector field~\cite{Dulac}. The difficult cases arise
when $\lambda\leq0$, and the really difficult cases (the ones that
still seem out of reach) occur when $\lambda$ is a negative irrational.
In the sequel we suppose that $\lambda\in\qq_{<0}$. 
\begin{defn}
We describe the class of \textbf{resonant }vector fields $Z$ and
their underlying resonant foliations $\fol Z$, assuming none of which
can be put in a linear form by conjugacy. We say that $Z$ admits
a \textbf{$p:q$ saddle} singularity at that point ($Z$ is a $p:q$
resonant vector field) if its eigenratio is $\lambda=-\frac{p}{q}$
for coprime positive integers $p$ and $q$. If $\lambda=0$ then
$Z$ is a \textbf{saddle-node} vector field.
\end{defn}

\begin{rem}
There exists a deep link between resonant saddles and saddle-nodes,
as we explain in Section~\ref{subsec:Summability}.
\end{rem}

The pioneering works of H.~\noun{Poincaré} and H.~\noun{Dulac} eventually
yield the formal classification of all $p:q$ resonant vector fields:
a codimension-$k$ vector field $Z$, for $k\in\zz_{>0}$, is always
formally conjugate to one of the \textbf{formal normal forms} \textbf{$P\left(u\right)\widehat{X}_{0}$
}where $u:=x^{q}y^{p}$ is called the \textbf{resonant monomial},
$P$ is a polynomial of degree at most $k$ in the variable $u$ with
$P\left(0\right)\neq0$, and 
\begin{eqnarray}
\widehat{X}_{0}\left(x,y\right) & := & u^{k}x\pp x+\left(1+\mu u^{k}\right)\left(\lambda x\pp x+y\pp y\right),~~\mu\in\cc.\label{eq:std_formal_NF}
\end{eqnarray}
 By this we mean that there exists an invertible formal power series
$\Psi=\left(\Psi_{1},\Psi_{2}\right)$, with $\Psi_{j}\in\frml{x,y}$,
such that $\Psi^{*}Z=P\widehat{X}_{0}$. This form is unique up to
the action of linear changes of variables $\left(x,y\right)\mapsto\left(\alpha x,\beta y\right)\in\GL 2{\cc}$
with $\left(\alpha^{q}\beta^{p}\right)^{k}=1$. The couple $\left(k,\mu\right)$
is the \textbf{formal orbital modulus} coming from the Dulac-Poincaré
normal form\textbf{~}\cite{Dulac}, while $P$ is the additional
\textbf{formal temporal modulus} obtained by A.~\noun{Bruno}~\cite{Brjuno}.
The formal modulus $\left(k,\mu,P\right)$ is left unchanged under
formal changes of variables on $Z$: it is a (complete) formal invariant.

We introduce the functional space of germs of a holomorphic function
in two complex variables $u$ and $y$: 
\begin{eqnarray*}
\polg uy & := & \left\{ y\sum_{n=1}^{2k}u^{n}f_{n}\left(y\right)~~~:~f_{n}\in\germ y\right\} ,
\end{eqnarray*}
where $\germ y$ is the algebra of germs of a holomorphic function
at $0\in\cc$. Consider for every parameter $c>0$ (Écalle's twist)
the polynomial vector field 
\begin{align}
X_{0}\left(x,y\right) & :=u^{k}x\pp x+\left(c\left(1-u^{2k}\right)+\mu u^{k}\right)Y\left(x,y\right)\label{eq:formal_NF}
\end{align}
where 
\begin{align*}
Y\left(x,y\right) & :=-px\pp x+qy\pp y.
\end{align*}
This is the main result of the article.
\begin{namedthm}[Normalization Theorem]
Let $Z$ be a $p:q$ resonant vector field with formal modulus $\left(k,\mu,P\right)$.
There exist:
\begin{itemize}
\item a bound $c\left(Z\right)>1$ and, for each choice of the twist $c\geq c\left(Z\right)$,
\item two germs $G,R\in\polg uy$ ,
\item a local holomorphic change of coordinates $\Psi\in\diff$,
\end{itemize}
such that 
\begin{align}
\Psi^{*}Z=Z_{G,R} & :=\frac{P}{1+PG}X_{R}\label{eq:analytic_NF}\\
X_{R} & :=X_{0}+RY.\nonumber 
\end{align}
Moreover any two $Z_{G,R}$ and $Z_{\widetilde{G},\widetilde{R}}$
are analytically conjugate near $\left(0,0\right)$ if and only if
they are conjugate by some $\left(x,y\right)\mapsto\left(\alpha x,\beta y\right)\in\GL 1{\cc}$
with $\left(\alpha^{q}\beta^{p}\right)^{k}=1$.
\end{namedthm}
\begin{rem}
~
\begin{enumerate}
\item The reader should be aware at that point of a slight abuse of notation.
When plugging the variables $\left(x,y\right)$ in the expression
above, the monomial $u$ should be substituted with $u\left(x,y\right)$.
For instance if $\mu=0$, then
\begin{align*}
X_{0}\left(x,y\right) & =y^{pk}x^{qk+1}\pp x+c\left(1-x^{2qk}y^{2pk}\right)\left(-px\pp x+qy\pp y\right).
\end{align*}
\item The proof actually asserts that the $2k$ functions $y\mapsto f_{n}\left(y\right)$
appearing in $R$ and $G$ are holomorphic and bounded on the disc
$\left\{ y~:~\left|y\right|<2\right\} $.
\end{enumerate}
\end{rem}

\begin{rem}
~
\begin{enumerate}
\item In~\cite{Loray} it is proved that any germ of a saddle foliation
at $\left(0,0\right)$ with eigenratio $\lambda$ can be defined in
a convenient local analytic chart by a vector field of the form
\begin{eqnarray*}
x\pp x+\lambda\left(f\left(y\right)+x\right)y\pp y &  & ~~,~f\in1+y\germ y.
\end{eqnarray*}
Clearly this form is very simple, but it is not unique. The link between
this form and the normal forms presented here is unclear.
\item In~\cite{EcalReal} a methodological approach for ``canonical spherical
synthesis'' of \emph{e.g.} resonant foliations is proposed. Unfortunately
it is not possible to directly extract from it an explicit form for
the synthesized vector fields.
\item The form of $R$ and $G$ is satisfying and seems optimal in the sense
that there is as many free ($2k$) components in $R$ (\emph{resp}.
$G$) than there is in the orbital modulus (\emph{resp}. temporal
modulus). Although the mapping 
\begin{align*}
\text{modulus~~}\longmapsto\text{~~normal~form }
\end{align*}
 is certainly not as simple as in the case of unilateral moduli described
in~\cite{SchaTey}, where it is ``triangular'' and computable,
the works of Écalle may provide a path to find an explicit expression
for it.
\end{enumerate}
\end{rem}

\begin{defn}
We use the notations introduced in the Normalization Theorem.
\begin{enumerate}
\item The vector field $Z_{G,R}$ is called an \textbf{analytic normal form}
of $Z$.
\item The vector field $X_{R}$ is called an \textbf{analytic orbital normal
form} of $Z$. The name is justified by the fact that $\fol Z$ is
analytically conjugate to $\fol{X_{R}}$.
\end{enumerate}
\end{defn}

\subsection{Outline of the construction and structure of the article}

The proof of the Normalization Theorem is done in three steps and
relies on Martinet-Ramis orbital classification of resonant planar
foliations~\cite{MaRa-Res}, which is summarized in Section~\ref{sec:Martinet-Ramis}.
That general scheme has already been used successfully in~\cite{SchaTey,RT2},
although the technical intricacies differ from one case to the other
and require specific arguments. In particular in the present situation
we rely on results that have recently been obtained in~\cite{TeySphere}
for the realization of analytic class of parabolic germs. (Of course
the link between the class of a saddle foliation and its holonomy
is well known, but here we do not directly invoke such arguments.) 
\begin{lyxlist}{00.00.0000}
\item [{\textbf{Orbital~realization}}] Being given the Martinet-Ramis
modulus associated to a resonant vector field $Z$, we build a vector
field $X_{R}$ in normal form within the same formal class and with
the same orbital modulus. This is done in Section~\ref{sec:Orbital_realization}
by a fixed-point method involving a Cauchy-Heine transform solving
a nonlinear Cousin problem associated to a sectorial decomposition
of the $\left(u,y\right)$-space (Section~\ref{sec:Notations}).
The trickiest part in the construction is to find a model vector field
$X_{0}$ whose orbit space in the intersection of consecutive sectors
can be controlled. More precisely, by increasing the twist parameter
$c$ we are able to shrink the size of the orbit spaces so that they
become adapted to the given orbital modulus. This is explained in
Section~\ref{sec:Martinet-Ramis}. Without the introduction of the
twist parameter it seems very dubious to provide normal forms: if
one tries by another mean to reduce the size of the orbit space in
one intersection to make it fit within the disc of convergence of
a component of the modulus, then the size increases in the next intersection
and may spill out of that of the corresponding component. The notable
exception comes from unilateral moduli, where only one out of two
components are nontrivial and the same strategy as~\cite{SchaTey}
would work for resonant saddle vector fields.\\
The process yields a vector field $X_{R}$ where $R$ is holomorphic
on some «hollow» domain, which is not a neighborhood of the origin
but contains the tube $\cc\times\left\{ 1<\left|y\right|<2\right\} $.
By studying the growth of $x\mapsto R\left(x,y\right)$ for fixed
$y$ we deduce that $R$ has a polynomial form $\sum_{n=1}^{2k}u^{n}f_{n}\left(y\right)$,
while the shape of the hollow domain forces each $f_{n}$ to extend
holomorphically to the whole disc $\left\{ \left|y\right|<2\right\} $.
\item [{\textbf{Temporal~realization}}] So far we have found an analytic
change of coordinates bringing $Z$ to some $UX_{R}$ with $U\left(0,0\right)\neq0$.
Sending $UX_{R}$ to $Z_{G,R}=\frac{P}{1+PG}X_{R}$ is done by solving
the cohomological equation $X_{R}\cdot T=\frac{1}{U}-\frac{1}{P}-G$,
where $X_{R}\cdot T$ is the Lie (directional) derivative of the function
$T$ along $X_{R}$. Being given $U$ and $P$, there is a unique
choice of $G$ in normal form such that the equation has an analytic
solution $T$. To understand this we need to exhibit an explicit cokernel
for the derivation $X_{R}$, by providing a section of the period
operator associated to $X_{R}$. Here again the main tool is the Cauchy-Heine
transform (although it does not need to be iterated). This study is
performed in Section~\ref{sec:Period} but we give more details in
Section~\ref{subsec:Period_section} to come, since knowing the cokernel
is worthwhile and carries a lot of useful applications.
\item [{\textbf{Uniqueness~of~the~realization}}] So far the family of
vector fields $\left\{ Z_{G,R}\right\} $ has been proved versal,
thus it is natural to describe the automorphisms of the family to
study its universality. Once the diagonal action $\left(x,y\right)\mapsto\left(\alpha,\beta y\right)$
with $\left(\alpha^{q}\beta^{p}\right)^{k}=1$ has been factored out,
it remains to prove that the only automorphism which is tangent to
the identity is the identity itself. According to the discussion in
the previous item, the function $G$ is unique for a given orbital
class $X_{R}$, therefore we must prove that if two foliations in
normal form $\fol{X_{R}}$ and $\fol{X_{\widetilde{R}}}$ are conjugate
by some $\Psi$ then $\Psi=\id$. First, we show that $\Psi$ can
be assumed to preserve the resonant monomial, \emph{i.e.} $\Psi=\left(x\ee^{-pN},y\ee^{qN}\right)$
for some holomorphic germ $N$, then we relate the condition $\Psi^{*}\fol{X_{R}}=\fol{X_{\widetilde{R}}}$
to a cohomological equation involving $N$. The latter has only trivial
solutions, thanks to the description of the cokernel that we have
done. We give a full proof of the uniqueness statement in Section~\ref{sec:Uniqueness}.
\end{lyxlist}

\subsection{\label{subsec:Period_section}Cohomological equations, period operator
and its natural section}

Because $\left[X_{0},Y\right]=0$ we can follow the strategy laid
out in~\cite{Tey-SN}.
\begin{itemize}
\item $X_{0}$ is (formally/analytically) conjugate to $X_{R}$ if and only
if there exists a (formal/analytic) solution of the cohomological
equation
\begin{eqnarray*}
X_{R}\cdot N & = & -R.
\end{eqnarray*}
In that case a conjugacy is obtained as $\flow YN{~:~\left(x,y\right)\mapsto\left(x\exp\left(-pN\left(x,y\right)\right),y\exp\left(qN\left(x,y\right)\right)\right)}$,
the flow at time $N$ of the vector field $Y$.
\item $UX_{R}$ is (formally/analytically) conjugate to $VX_{R}$ if and
only if there exists a (formal/analytic) solution of the cohomological
equation
\begin{eqnarray*}
X_{R}\cdot T & = & \frac{1}{U}-\frac{1}{V}.
\end{eqnarray*}
In that case a conjugacy is obtained as $\flow{VX_{R}}T{}$.
\end{itemize}
Such equations are called \textbf{cohomological equations}\emph{.
}The obstructions to solve formally or analytically these equations
reasonably provides us with invariants of classification. 
\begin{namedthm}[Cohomological Theorem]
Let $X_{R}$ be an orbital normal form.
\begin{enumerate}
\item Consider the cohomological equation $X_{R}\cdot F=G$ with $G\in\germ{x,y}$.
\begin{enumerate}
\item There exists a formal solution $F\in\frml{x,y}$ if and only if the
Taylor expansion of $G$ at $\left(0,0\right)$ does not contain terms
$u^{n}$ for $n\in\left\{ 0,\ldots,k\right\} $ (in that case $F$
is unique). We write $\germ{x,y}_{>k}$ the space of all such germs,
and we assume in the following that $G$ belong to that space.
\item There exists a neighborhood $\Omega$ of $\left(0,0\right)$ on which
$G$ is holomorphic and bounded, a covering of $\Omega\backslash\left\{ xy=0\right\} $
into $2k$ ``sectors'' $\obj[j][\sharp]{\Omega}$, for $j\in\zsk$
and $\sharp\in\left\{ \zi,\iz\right\} $, together with $2k$ bounded
holomorphic functions $\obj F\in\holb[{\obj{\Omega}}]$, such that
$X_{R}\cdot\obj F=G$.
\item The difference $\obj[][\zi]F-\obj[][\iz]F$ (\emph{resp}. $\obj[j+1][\iz]F-\obj[][\zi]F$)
is constant on the leaves of $\fol{X_{R}}$ and tends to $0$ on $\left\{ xy=0\right\} $,
therefore it defines a holomorphic function $\obj[][-]f\in\germ{\frac{1}{h}}_{>0}$
(\emph{resp}. $\obj[][+]f\in\germ h_{>0}$) of the leaf coordinate
$h\in\cc$.
\end{enumerate}
\item We call \textbf{period operator} of $X_{R}$ the linear mapping
\begin{eqnarray*}
\per\,:\,G\in\germ{x,y}_{>k} & \longmapsto & \left(\obj[][-]f,\obj[][+]f\right)_{j\in\zsk}\in\left(\germ{\frac{1}{h}}_{>0}\times\germ h_{>0}\right)^{\times k}.
\end{eqnarray*}

\begin{enumerate}
\item The formal solution $F$ given in 1.(a) is a convergent power series
if and only if $\per\left(G\right)=0$.
\item The period operator is surjective. More precisely, being given $f\in\left(\germ{\frac{1}{h}}_{>0}\times\germ h_{>0}\right)^{\times k}$
there exists a unique $G\in\polg uy$ such that $\per\left(G\right)=f$. 
\end{enumerate}
\end{enumerate}
\end{namedthm}
\begin{defn}
The operator $\sec R~:~f\mapsto G$ defined in 2. of the Cohomological
Theorem is called the \textbf{natural section of the period operator}
of $X_{R}$.
\end{defn}

We may reformulate algebraically the previous theorem by saying that
the following sequence of linear maps is exact:
\[
\begin{array}{ccccccc}
 & \mbox{cst} &  & X_{R}\cdot &  & \per\\
\ww C & \longinj{}{} & \germ{x,y} & \longrightarrow & \germ{x,y}_{>k} & \longsurj{}{} & \left(\germ{\frac{1}{h}}_{>0}\times\germ h_{>0}\right)^{\times k}
\end{array}
\]
with a similar exact sequence at a formal level
\[
\begin{array}{ccccccc}
 & \mbox{cst} &  & X_{R}\cdot &  & \Pi_{k}\\
\ww C & \longinj{}{} & \frml{x,y} & \longrightarrow & \frml{x,y} & \longsurj{}{} & \pol u_{\leq k}
\end{array},
\]
where $\Pi_{k}~:~\sum_{n,m\in\zp}g_{n,m}x^{n}y^{m}\mapsto\sum_{\ell\leq k}g_{q\ell,p\ell}u^{\ell}$
is the canonical projection coming from the power series expansion.
We deduce from this theorem the following interpretation of the different
moduli involved (see Corollary~\ref{cor:invar_as_period}).
\begin{itemize}
\item The obstruction to solve formally $X_{R}\cdot\widehat{T}=\frac{1}{U}$
is located in $P:=\Pi_{k}\left(U\right)$ and that gives the formal
temporal modulus of Bruno.
\item The obstruction to solve analytically $X_{R}\cdot T=\frac{1}{U}-\frac{1}{P}$
is embodied the period $\per\left(\frac{1}{U}-\frac{1}{P}\right)$
and that gives the temporal modulus ($t$-shift) of Grintchy-Voronin.
\item The obstruction to solve analytically $X_{R}\cdot N=-R$ is the period
$\per\left(-R\right)$ and that provides the logarithmic form of Martinet-Ramis
orbital modulus.
\end{itemize}
Partial results pertaining to the Cohomological~Theorem (namely,
2.~(a)) were already obtained in~\cite{BerLor}.
\begin{rem}
Due to the particular structure of the leaf space of $\fol{X_{R}}$,
the period operator is not onto the whole space of germs $\left(\germ{\frac{1}{h}}_{>0}\times\germ h_{>0}\right)^{\times k}$.
Indeed, by taking a smaller neighborhood of $\left(0,0\right)$ the
size of the leaf space in the intersection of consecutive sectors
does not shrink to a point. To realize a given element of that space
as a period of $X_{R}$ it is probably necessary to take a larger
$c$ (and thus another $R$ while staying in the same orbital class).
Precise statements are given in Theorem~\ref{thm:natural_section}.
\end{rem}

As an application of the Cohomological Theorem, the same reasoning
as the one produced in~\cite{RT2} (that uses the natural section
of the period operator) allows to generalize a result of M.~\noun{Berthier}
and F.~\noun{Touzet}~\cite[Proposition 5.5]{BerTouze} to resonant
saddle foliations: a resonant saddle foliation admits a Liouvillian
first-integral if and only if its orbital normal form $X_{R}$ is
a Bernoulli vector field, that is $R\left(u,y\right)=y^{d}r\left(u\right)$
for some $d\in\zp$ and $r\in u\pol u_{<2k}$. We leave details to
the interested reader.

\subsection{\label{subsec:Summability}Summability and divergence}

The Cohomological Theorem offers in 1.(b) ``sectorial'' sums to
the (generally) divergent power series $F$. As will be made clear
in Section~\ref{sec:Period}, the divergence is concentrated in the
resonant monomial. This property was already underlined in~\cite{MaRa-Res}
for the orbital problem. In fact, when $G\in\germ{u,y}_{>k}$ one
can prove that the $\obj F$ come from holomorphic functions in the
variable $\left(u,y\right)\in\obj V\times\left\{ \left|y\right|<2\right\} $
where $u$ is replaced by $u\left(x,y\right)$ and $\sect$ is a traditional
sector in the $u$-variable. From Ramis-Sibuya's theorem (see \emph{e.g.~}\cite{LodRich}),
we can deduce that $F=\sum_{n\geq0}f_{n}\left(y\right)u^{n}$ where
the $f_{n}$ are holomorphic and bounded on the disc $D:=\left\{ \left|y\right|<2\right\} $
with $\norm[f_{n}]D=\OO{B^{n}\left(n!\right)^{\nf 1k}}$, \emph{i.e.}
that $F$ is transversely $k$-summable in the variable $u$.

All these facts can also be deduced from corresponding properties
already known for saddle-node vector fields. One may indeed observe
that the resonant saddle vector fields we obtain as normal forms do
come from saddle-node vector fields in the variables $\left(u,y\right)$.
This is made apparent by considering the foliation $\fol{X_{R}}$
as integrating the distribution of dual differential $1$-forms:
\begin{align*}
u^{k}x\dd y+\left(c\left(1-u^{2k}\right)+\mu u^{k}+y\sum_{n=1}^{2k}u^{n}f_{n}\left(y\right)\right)\left(px\dd y+qy\dd x\right).
\end{align*}
Recalling that $u=x^{q}y^{p}$, we deduce that $\fol{X_{R}}$ also
integrates the distribution given by
\begin{align}
\omega_{R}\left(u,y\right) & :=u^{k+1}\dd y+y\left(c\left(1-u^{2k}\right)+\mu u^{k}+y\sum_{n=1}^{2k}u^{n}f_{n}\left(y\right)\right)\dd u.\label{eq:NF_u-y}
\end{align}
This is exactly a saddle-node in the variables $\left(u,y\right)$
with formal invariant $\left(k,\mu\right)$. This correspondence is
well known for formal normal forms ($R:=0$), and we just established
it at an analytical level. 

\section{\label{sec:Notations}Notations and basic tools}

Since a lot of objects of different natures mix up (germs at $\left(0,0\right)$
of holomorphic objects, sectorial objects, Banach spaces of functions
\emph{etc}.), we must introduce numerous notations. This section provides
the reader with a glossary of notations and conventions we stick to
throughout the article. 
\begin{itemize}
\item It will be convenient to follow the convention 
\begin{align*}
0^{+1}:=0~~~ & \text{and}~~~0^{-1}:=\infty~.
\end{align*}
\item We fix a pair $\left(p,q\right)$ of coprime positive integers and
we define the associated \textbf{resonant monomial}
\begin{align*}
u\left(x,y\right) & :=x^{q}y^{p}.
\end{align*}
\item Most constructions take place in the variables $\left(u,y\right)$
and are pullbacked in the variables $\left(x,y\right)$ through the
\textbf{canonical embedding} 
\begin{align*}
\iota~ & :~\left(x,y\right)\longmapsto\left(u\left(x,y\right),y\right).
\end{align*}
In order to keep notations as light as possible, we write $u_{*}$
to stand for the value of $u\left(x,y\right)$ in expressions containing
$x$ and $y$, in order to distinguish it from the usage of $u$ as
a standalone symbol. We use a similar notation for functions. For
instance starting from $f~:~\left(u,y\right)\mapsto f\left(u,y\right)$
we write $f_{*}$ to stand for the function $\iota^{*}f=f\circ\iota$:
\begin{align*}
f_{*}~:~\left(x,y\right) & \longmapsto f\left(u\left(x,y\right),y\right).
\end{align*}
\end{itemize}

\subsection{\label{subsec:sectors}Sector-related notations}

We fix $k\in\zz_{>0}$ . Undoubtedly the biggest source of notational
heaviness comes from the decomposition of the $\left(x,y\right)$-space
and $u$-space into $2k$ sectors. This decomposition is classical
yet we wish to introduce notations that both contain all necessary
contextual information and embodies the underlying dynamical structure.
The vast majority of objects we introduce are collections $O$ of
$2k$ sectorial objects $\obj[][\bullet]O$ indexed by $j\in\zsk$
and $\bullet\in\left\{ +,-,\zi,\iz\right\} $. Here is the list of
the conventions that are always used in the rest of the article:
\begin{itemize}
\item $j$ is some (arbitrary) element of $\zsk$ ;
\item the symbol $\sharp$ is some (arbitrary) element of $\left\{ \zi~,~\iz\right\} $;
\item we simply write the collection $O=\left(\obj O\right)_{j\in\zsk,\,\sharp\in\left\{ \zi,\iz\right\} }$
as 
\begin{eqnarray*}
O & = & \left(\obj O\right),
\end{eqnarray*}
and the collection $O=\left(\obj[][\star]O\right)_{j\in\zsk,\,\star\in\left\{ -,+\right\} }$
as 
\begin{eqnarray*}
O & = & \left(\obj[][\pm]O\right).
\end{eqnarray*}
\end{itemize}
\begin{figure}[H]
\hfill{}\includegraphics[width=0.4\columnwidth]{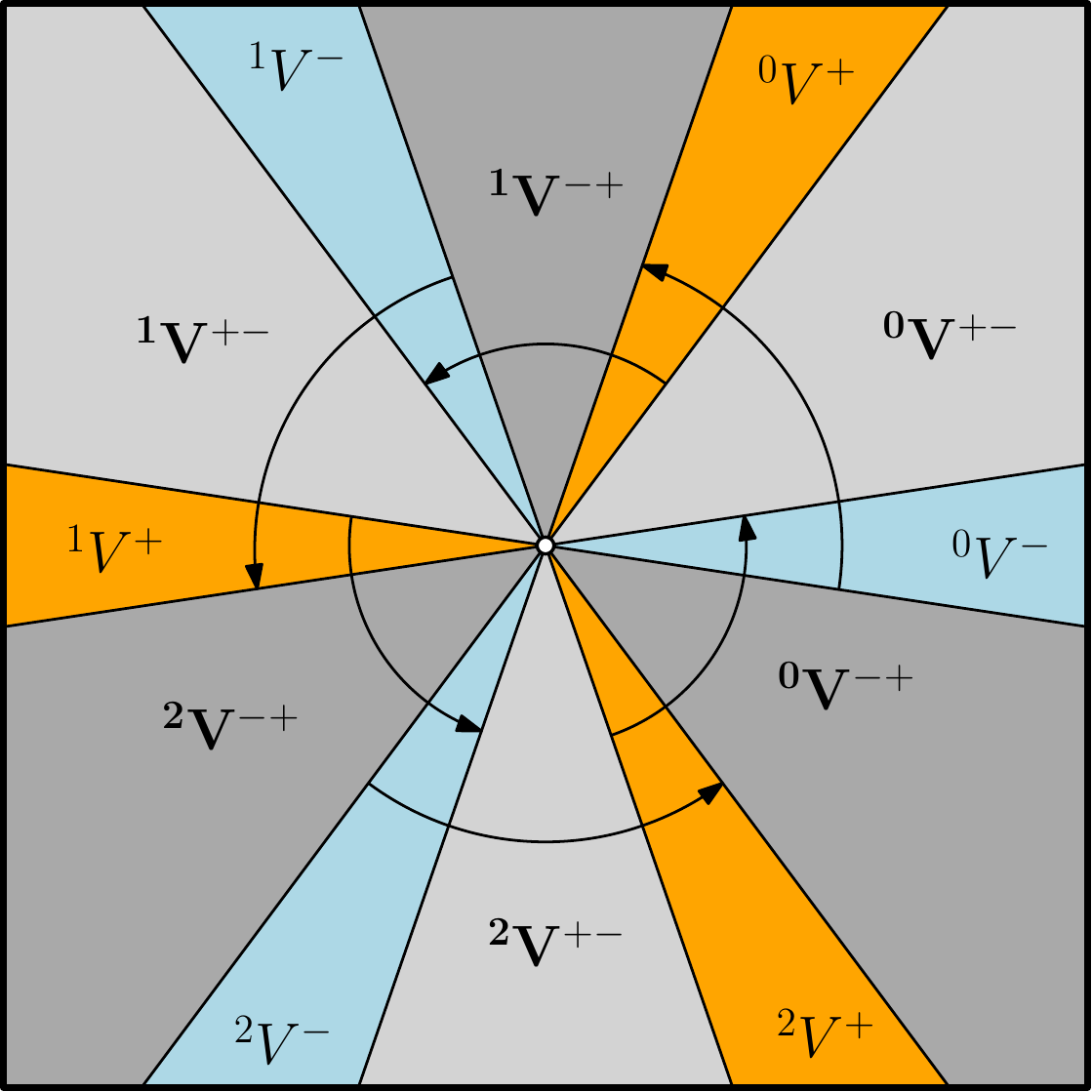}\hfill{}\includegraphics[width=0.4\columnwidth]{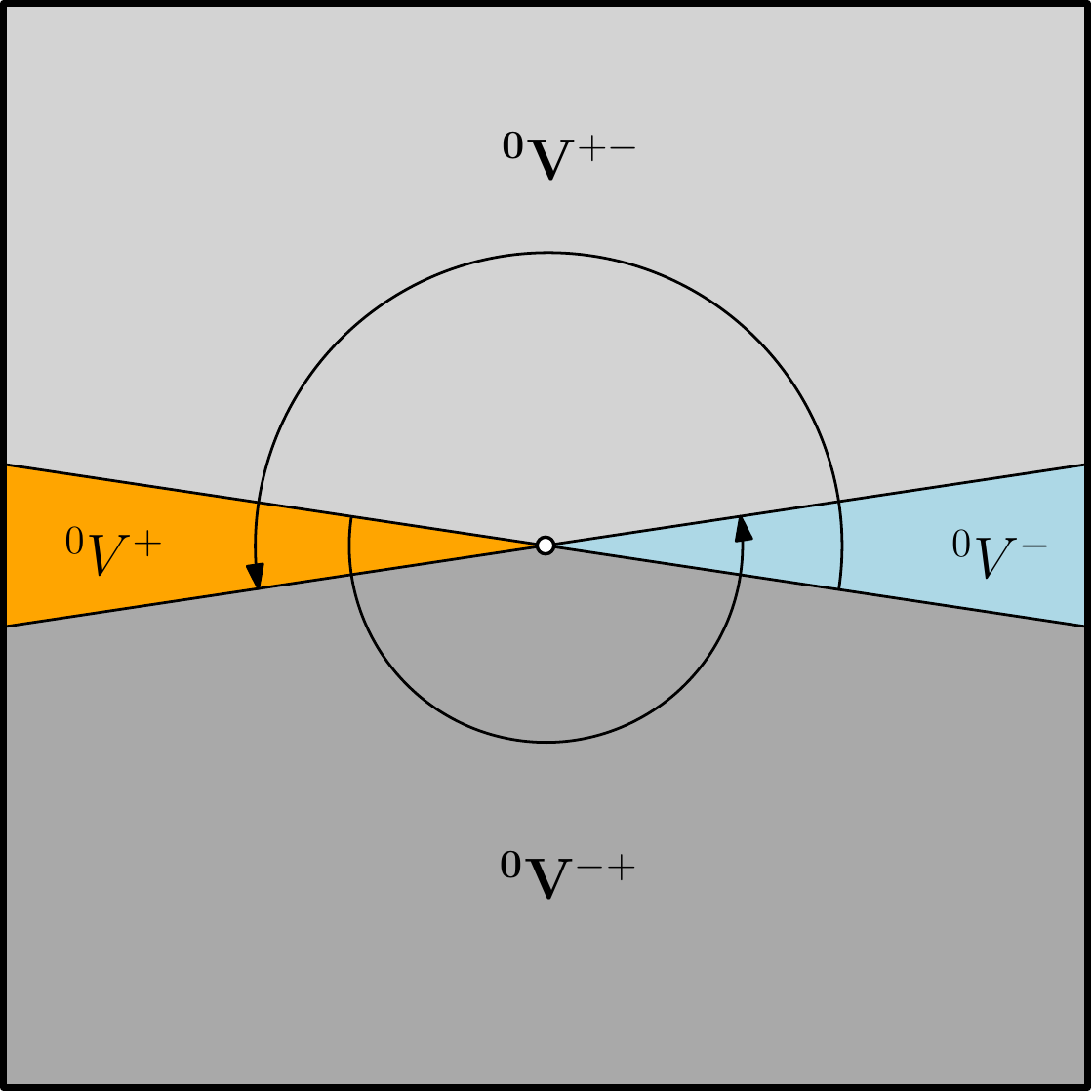}\hfill{}

\caption{\label{fig:sectors}The sectorial decomposition near $0$ in the case
$k=3$ (left) or $k=1$ (right).}
\end{figure}

\begin{defn}
\label{def_sectors}This should be read with the figure \ref{fig:sectors}
in mind. 
\begin{enumerate}
\item The\emph{ }\textbf{sectorial decomposition} of the $u$-space is the
collection of $2k$ open, infinite sectors $\left(\sect\right)$ defined
as
\begin{eqnarray*}
\sect[j][\zi] & := & \left\{ u\neq0\,:\,\left|\arg u-\pi\frac{2j+1}{2k}\right|<\frac{5\pi}{8k}\right\} ,\\
\sect[j][\iz] & := & \left\{ u\neq0\,:\,\left|\arg u+\pi\frac{2j+1}{2k}\right|<\frac{5\pi}{8k}\right\} .
\end{eqnarray*}
\item The intersection of pairwise successive sectors is either one of the
following sectorial domains (if $k>1$) 
\begin{align*}
\sect[][-] & :=\sect[][\zi]\cap\sect[][\iz],\\
\sect[][+] & :=\sect[j+1][\iz]\cap\sect[][\zi].
\end{align*}
We extend this definition for $k=1$ by considering the two connected
components $\sect[0][-]$ and $\sect[0][+]$ of $\sect[0][\zi]\cap\sect[0][\iz]$
under the condition that $\sect[0][\pm]$ contains $\pm\rr_{<0}$.
\item Let us define:
\begin{align*}
\ssect[][\zi] & :=\left\{ \left(u,y\right)\in\cc^{2}~:~\left|y\right|<2,~u\in\sect[][\zi]\right\} \\
\ssect[][\iz] & :=\left\{ \left(u,y\right)\in\cc^{2}~:~1<\left|y\right|<2,~u\in\sect[][\iz]\right\} ,
\end{align*}
as well as the corresponding intersections $\ssect[][\pm]$ build
in a similar fashion as in~2. The notation $\mathcal{V}$ denotes
the full collection $\left(\ssect\right)_{j,\sharp}$.
\item We pull-back these sectors by defining 
\begin{align*}
\ssect_{*} & :=\iota^{-1}\left(\ssect\right)=\left\{ \left(x,y\right)\in\cc^{2}~:~\iota\left(x,y\right)\in\ssect\right\} ,
\end{align*}
 and form the filled union
\begin{align*}
\mathcal{V}_{*} & :=\tx{int}\left(\overline{\bigcup_{j,\sharp}\ssect_{*}}\right).
\end{align*}
\end{enumerate}
\end{defn}

\begin{rem}
The domain $\mathcal{V}_{*}$ is the disjoint union of $\bigcup_{j,\sharp}\ssect_{*}$
on the one hand and $\left\{ \left(x,y\right)\in\mathcal{V}_{*}~:~xy=0\right\} $
on the other hand. It is not a neighborhood of $\left(0,0\right)$,
and each $\ssect_{*}$ is not connected whenever $\left(p,q\right)\neq\left(1,1\right)$.
\end{rem}

\subsection{Function spaces}

We write $\pol z_{d\geq\nu}$ the algebra of polynomials in $z$ of
degree at most $d$ and valuation at least $\nu$. We omit to write
$\nu$ whenever it equals $0$. The field $\cc$ may also be replaced
with other commutative rings.
\begin{defn}
\label{def_Banach_space}~
\begin{enumerate}
\item Let $\mathcal{D}\subset\ww C^{n}$ be a domain. We denote by $\holf[\mathcal{D}]$
the algebra of functions holomorphic on $\mathcal{D}$.
\item We define the Banach algebra $\holb[\mathcal{D}]$ of all $\cc$-valued
bounded holomorphic functions on $\mathcal{D}$ with continuous extension
to the closure $\overline{\mathcal{D}}$, equipped with the norm:
\begin{eqnarray*}
\norm[f]{\mathcal{D}} & := & \sup_{\mathbf{z}\in\mathcal{D}}\left|f\left(\mathbf{z}\right)\right|.
\end{eqnarray*}
\item Being given a finite collection $\mathcal{D}:=\left(\obj[][\pm]{\mathcal{D}}\right)$
of $2k$ domains of $\cc^{2}$, we denote by $\holb[\mathcal{D}]$
the product Banach space $\prod_{j,\pm}\holb[{\obj[][\pm]{\mathcal{D}}}]$
equipped with the product norm 
\begin{eqnarray*}
f=\left(\obj[][\pm]f\right)~,~~~~\norm[f]{\mathcal{D}} & := & \max_{j,\pm}\norm[{\obj[][\pm]f}]{\obj[][\pm]{\mathcal{D}}}.
\end{eqnarray*}
\item Let $D^{\pm}\subset\proj$ be a domain containing $0^{\pm1}$. We
define the Banach algebra $\holb[D]'$ of all $\cc$-valued bounded
holomorphic functions on $D$, admitting a continuous extension to
the closure $\overline{D}$ and vanishing at $0^{\pm1}$, equipped
with the norm
\begin{eqnarray*}
\norm[f]D' & := & \sup_{z\in D}\frac{\left|f\left(z\right)\right|}{\left|z\right|^{\pm1}}.
\end{eqnarray*}
If $D$ is a star-shaped domain centered at $0^{\pm1}$, then $\norm[f]D'\leq\norm[f']D$
.
\end{enumerate}
\end{defn}

\subsection{Vector fields}
\begin{itemize}
\item $Z$ is a resonant-saddle or saddle-node vector field near $\left(0,0\right)$.
The notation $X$ is in general reserved for vector fields with a
trivial temporal component.
\item $Z\cdot F$ stands for the \textbf{Lie derivative} along $Z$, acting
on $F\in\frml{x,y}$ or on $F\in\germ{x,y}$.
\item We let $\flow Zt{\left(x,y\right)}$ be the \textbf{flow} at time
$t$ of $Z$, \emph{i.e}. the local holomorphic solution of~\ref{eq:flow_system}
with initial condition $\left(x,y\right)$. It is locally holomorphic
in the variables $\left(x,y,t\right)$ taken sufficiently close to
$\left(0,0,0\right)$.
\item A \textbf{first-integral} $H$ of $Z$ is a holomorphic function such
that $Z\cdot H=0$.
\item $\left(k,\mu\right)\in\ww N_{>0}\times\ww C$ is the \textbf{formal
orbital modulus} of $Z$ while $P\in\pol u_{\leq k}$ with $P\left(0\right)\neq0$
is its \textbf{formal temporal modulus}. The complete formal modulus
is $\left(k,\mu,P\right)$.
\item The \textbf{formal orbital normal form} associated to the formal modulus
$\left(k,\mu,P\right)$ is the polynomial vector field depending on
the twist parameter $c>1$:
\begin{align}
X_{0}\left(x,y\right) & :=u_{*}^{k}x\pp x+\left(c\left(1-u_{*}^{2k}\right)+\mu u_{*}^{k}\right)Y\left(x,y\right)~\label{eq:orbital_formal_model}\\
Y\left(x,y\right) & :=-px\pp x+qy\pp y.\nonumber 
\end{align}
\end{itemize}

\section{\label{sec:Martinet-Ramis}Martinet-Ramis orbital modulus of a resonant
saddle}

Finding formal models for the dynamics of 2-dimensional vector fields
is easy enough. When these formal normalizations fail to be analytic,
one must perform a finer study to obtain the analytical classification.
In the orbital case, \emph{i.e.} that of foliations, this amounts
to endowing the leaf space with a holomorphic structure and describing
the analytic diffeomorphisms between these manifolds. The now-classical
strategy for resonant foliations is to build adapted sector-like areas
whose closure is a neighborhood of the singularity, and to find normalizing
sectorial maps conjugating the dynamics with that of the formal model
$\fol{X_{0}}$. 

For expository reasons, in this section we briefly explain how both
tasks are achieved for $1:1$ foliations, following the ideas of J.~\noun{Martinet}
and J.-P.~\noun{Ramis}~\cite{MaRa-SN,MaRa-Res} and introducing
some material needed later on. In Section~\ref{subsec:Realization_p:q}
we stress the slight modifications that are needed to make the general
theory for $p:q$ resonant foliations work. 

\subsection{\label{subsec:Model_orbit-space}Study of the formal model and making
of the sectors}

Here we investigate the global dynamical properties, for a fixed value
of $c>0$, of the vector field $X_{0}$ given by~(\ref{eq:orbital_formal_model}).
We are particularly interested in describing its orbit space, which
can be achieved through the study of the Liouvillian first-integral
\begin{align}
H\left(u,y\right) & :=y\widehat{H}\left(u\right),\nonumber \\
\widehat{H}\left(u\right) & :=u^{-\mu}\exp\frac{c\left(u^{-k}+u^{k}\right)}{k}.\label{eq:model_first-int}
\end{align}
By letting $X_{0}\cdot$ stand for the Lie directional derivative
along $X_{0}$, an elementary computation yields
\begin{align*}
X_{0}\cdot H_{*} & =0.
\end{align*}
(In fact $X_{0}$ is built as the dual vector field of the rational
$1$-form $\frac{\dd H}{H}$.) This identity tells us that level sets
of $H_{*}$ coincide with trajectories of $X_{0}$. That is, an equation
of a leaf of $\fol{X_{0}}$ is given by 
\begin{align*}
H_{*}\left(x,y\right) & =\cst.
\end{align*}
Yet, because $\widehat{H}$ is multivalued when $\mu\notin\zz$, some
care needs to be taken; in the sequel we use the determination of
the argument of $u$ on $\cc\backslash\rr_{\geq0}$ to compute the
actual value of $u^{-\mu}$: when crossing the boundary from $\sect[0][\iz]$
to $\sect[0][\zi]$ the value of $u^{-\mu}$ is multiplied with $\ee^{-2\ii\pi\mu}$.
We wish to distribute evenly this change over all sectors $\sect$.
\begin{defn}
\label{def:formal_model_first-integral}Set 
\begin{align*}
\sigma & :=\exp\frac{\ii\pi\mu}{k}.
\end{align*}
We define the \textbf{model }(sectorial) \textbf{first-integrals}
as $\hh[][][0]$ where: 
\begin{align*}
\hh[][][0]\left(u,y\right):= & \sigma^{n}y\times\widehat{H}\left(u\right)\,\,\text{with }n:=\begin{cases}
2j & \text{if }\sharp=\zi\\
2j-1 & \text{if }\sharp=\iz\,\text{and }j>0\\
2k-1 & \text{if }\sharp=\iz\,\text{and }j=0
\end{cases}.
\end{align*}
\end{defn}

The third item of the next lemma is the key property that allows the
rest of the construction to be worked out. The vector field $X_{0}$
has been designed so that it holds.
\begin{lem}
\label{lem:secto_first-int_NF}~
\begin{enumerate}
\item Let $\mathcal{V}$ be a small domain containing $\left(0,0\right)$
and take $M,\widetilde{M}\in\mathcal{V}\cap\ssect$. The identity
$\hh[][][0]_{*}\left(M\right)=\hh[][][0]_{*}\left(\widetilde{M}\right)$
holds if and only if there exists a leaf $\mathcal{L}$ of the restricted
foliation $\fol{X_{0}}|_{\mathcal{V}\cap\ssect}$ such that $M\in\mathcal{L}$
and $\widetilde{M}\in\mathcal{L}$.
\item $\hh[][][0]\left(\ssect\right)=\cc^{\times}$.
\item $\hh[][][0]\left(\ssect[][\pm]\right)$ is a punctured neighborhood
of $0^{\pm1}$. More precisely, there exists $A=A\left(k,\mu\right)>0$
and $c\left(k,\mu\right)>0$ such that if $c>c\left(k,\mu\right)$,
then:
\begin{align*}
\norm[\left(~_{0}^{j}\mathcal{H}^{\sharp}\right)^{\pm1}]{\ssect[][\pm]} & \leq A\ee^{-\nf ck}.
\end{align*}
\end{enumerate}
\end{lem}

\begin{proof}
Most of the assertions can be found in~\cite{MaRa-Res}. Especially~1.
is proved in~\cite[p593, p598]{MaRa-Res} for the usual first-integral
$\widehat{H}_{0}\left(x,y\right):=yu_{*}^{-\mu}\exp\nf{u_{*}^{-k}}k$
of the standard model $\widehat{X_{0}}$ (see~(\ref{eq:std_formal_NF})).
Using the fact that the critical points of $u^{k}\mapsto u^{k}+u^{-k}$
lie on the unit circle in the $u$-variable, we deduce our claim by
requiring that the size of $\mathcal{V}$ be so small as to ensure
$\norm[u]{\mathcal{V}}<1$. Item~2. comes from the fact that $u\in\sect\longmapsto\widehat{H}\left(u\right)$
admits an essential singularity at $0$ while the values reached by
$u^{k}+u^{-k}$ cover a punctured neighborhood of $0$.

The estimates appearing in~3. follow from elementary calculus. In~\cite[Corollary 4.5]{TeySphere}
a value for $c\left(k,\mu\right)$ is determined and a bound $\norm[\widehat{H}]{\sect[][\pm]}^{\pm1}<\mathfrak{m}\ee^{-\nf c{2k}}$
is proved for some explicit $\mathfrak{m}=\mathfrak{m}\left(k,\mu\right)$.
The cited paper deals with the case $k=1$ and the factor $u^{-\mu}$
is slightly different, but the general case follows from what has
been carried out there; details are left to the reader. If $\left|y\right|\geq1$
and $u\in\cc$, then 
\begin{align*}
\left|\frac{1}{\hh[][][0]\left(u,y\right)}\right| & =\frac{1}{\left|y\right|}\left|\sigma^{n}\widehat{H}\left(u\right)\right|^{-1}\leq\left|\sigma\right|^{-n}\left|\widehat{H}\left(u\right)\right|^{-1}~,~n\in\left\{ 0,\ldots,2k-1\right\} .
\end{align*}
A similar bound can be established when $\left|y\right|\leq2$ and
$u\in\cc$ since $\left|\hh[][][0]\left(u,y\right)\right|=\left|y\right|\left|\sigma^{n}\widehat{H}\left(u\right)\right|\leq2\left|\sigma\right|^{n}\left|\widehat{H}\left(u\right)\right|$
for some $n\in\left\{ 0,\ldots,2k-1\right\} $. The proof is complete.
\end{proof}
We can interpret~1. of the Lemma by saying that the values $h$ of
$\hh[][][0]_{*}$ provide a natural coordinate on the sectorial orbit
space of $X_{0}$ near $\left(0,0\right)$, and that~2. makes the
sectorial orbit space a punctured sphere $\cbar\backslash\left\{ 0,\infty\right\} $.
The discarded values $0$ and $\infty$ correspond to the two separatrices
$\left\{ xy=0\right\} \backslash\left\{ \left(0,0\right)\right\} $.
Taking~3. into account, we deduce that the orbit space of $X_{0}$
outside $\left\{ xy=0\right\} $ is obtained by identifying the $2k$
successive spheres about their poles by linear maps, since the choices
of determination of $H$ we made over the intersections $\sect[][\pm]$
imply:
\begin{align*}
\left(\forall\left(x,y\right)\in\ssect[j][-]_{*}\right)~~~~~\hh[][\zi][0]_{*}\left(x,y\right) & =\sigma\times\hh[][\iz][0]_{*}\left(x,y\right),\\
\left(\forall\left(x,y\right)\in\ssect[j][+]_{*}\right)~~\hh[j+1][\iz][0]_{*}\left(x,y\right) & =\sigma\times\hh[][\zi][0]_{*}\left(x,y\right).
\end{align*}
J.~\noun{Martinet} and J.-P.~\noun{Ramis} called this configuration
the \emph{chapelet de sphères} (rosary of spheres), which we prefer
to call the \emph{orbital necklace} of $\fol{X_{0}}$ as in~\cite{RT2}.
The orbital modulus of Martinet-Ramis is obtained in the case of a
general resonant foliation by replacing the linear polar identifications
with nonlinear perturbations.

\subsection{Sectorial normalization and sectorial first-integral}

Start with a $1:1$ resonant vector field $X_{R}=X_{0}+RY$ with $R$
holomorphic and $R\left(0,0\right)=0$ (following~\cite{Dulac,Dulac2}
any $1:1$ resonant saddle foliation can be brought into that form
by choosing suitable local analytic coordinates). According to~\cite[Theorem 6.2.1]{MaRa-Res},
there exists a neighborhood $\mathcal{V}$ of $\left(0,0\right)$
and a collection of functions $\mathcal{N}=\left(\obj{\mathcal{N}}\right)$
with $\obj{\mathcal{N}}\in\holb[\left\{ \left(x,y)\right)\in\mathcal{V}~:~u_{*}\in\sect\right\} ]$
such that, if one defines $\obj{\Psi}:=\flow Y{\obj{\mathcal{N}}}{}$,
then 
\begin{align*}
\left(\obj{\Psi}\right)^{*}X_{0} & =X_{R}.
\end{align*}

\begin{rem}
This result is a byproduct of Martinet-Ramis synthesis theorem. We
do not need it in our present study, we simply invoke it for the purpose
of our exposition of their classification. We revisit this assertion
in Section~\ref{sec:Period} by providing it with a more geometric
flavor.
\end{rem}

Because $X_{0}\cdot\hh[][][0]_{*}=0$ we have $X_{R}\cdot\left(\hh[][][0]_{*}\circ\obj{\Psi}\right)=0$,
where $\hh[][][0]_{*}$ is the first-integral of the formal model
$X_{0}$ defined in Section~\ref{subsec:Model_orbit-space}. Let
us describe in more details these sectorial first-integrals of $X_{R}$,
since they provide a natural coordinate on the sectorial leaf space
of $\fol{X_{R}}$.
\begin{lem}
\label{lem:secto_norma}The functions \textup{$\obj{\mathcal{N}}$
satisfy the following properties.}
\begin{enumerate}
\item $\hh[][][0]_{*}\circ\obj{\Psi}=\hh[][][0]_{*}\exp\obj{\mathcal{N}}$. 
\item $X_{R}\cdot\obj{\mathcal{N}}=-R$.
\item If $X_{R}$ is in normal form then $\obj{\mathcal{N}}=\obj N_{*}$
for some sectorial function in $\left(u,y\right)$-space.
\end{enumerate}
\end{lem}

\begin{proof}
~
\begin{enumerate}
\item Because $\obj{\Psi}\left(x,y\right)$ is given by the flow of $Y$
starting from $\left(x,y\right)$ and with time $\obj{\mathcal{N}}\left(x,y\right)$,
we have:
\begin{align*}
\obj{\Psi}\left(x,y\right) & =\left(x\exp\left(-\obj{\mathcal{N}}\left(x,y\right)\right),y\exp\left(\obj{\mathcal{N}}\left(x,y\right)\right)\right).
\end{align*}
In particular $u\circ\obj{\Psi}=u$, which gives the conclusion.
\item From the identities
\begin{eqnarray*}
Y\cdot u & = & 0\\
Y\cdot y & = & y,
\end{eqnarray*}
we derive the fact that $Y\cdot\hh[][][0]_{*}=\hh[][][0]_{*}$. Since
$X_{0}\cdot\hh[][][0]_{*}=0$ we conclude, by taking logarithmic derivatives:
\begin{eqnarray*}
\frac{\obj X\cdot\hh[][][N]}{\hh[][][N]} & = & \obj X\cdot\left(\log\hh[][][0]_{*}+\obj N\right)\\
 & = & R+\obj X\cdot\obj N=0
\end{eqnarray*}
as required.
\item See Corollary~\ref{cor:orbital_saddle}~1.
\end{enumerate}
\end{proof}
Hence, we are led to give the following definition.
\begin{defn}
\label{def_first-integral}~
\begin{enumerate}
\item Define the Banach space of functions in the variables $\left(u,y\right)$:
\begin{eqnarray*}
\obj{\mathcal{A}} & := & \holb[{\ssect}],
\end{eqnarray*}
as well as the product algebra $\mathcal{A}:=\prod_{j,\sharp}\obj{\mathcal{A}}$.
\item Let $N\in\mathcal{A}$ be a collection of $2k$ sectorial, bounded
and holomorphic functions. We define 
\begin{eqnarray*}
\hh & := & \hh[][][0]\times\exp\obj N\,.
\end{eqnarray*}
 The collection $\left(\hh_{*}\right)$ is called the \textbf{canonical
sectorial first-integral}\emph{ }associated to $N$.
\end{enumerate}
\end{defn}

Because of the choices made for $\hh[][][0]$, we have
\begin{align*}
\hh & \in\holf[{\ssect}].
\end{align*}
It is straightforward to show that the conclusions of Lemma~\ref{lem:secto_first-int_NF}
hold for these first-integrals but for the presence of the perturbations
$\obj N$. For the sake of brevity, the next proposition is only written
down for foliations in normal form~(\ref{eq:analytic_NF}), but it
can be adapted in a straightforward manner to encompass the case of
a general resonant vector field $X_{0}+RY$. We leave the details
to the reader.
\begin{prop}
\label{prop:secto_first-int}Let $N\in\mathcal{A}$.
\begin{enumerate}
\item Let $\mathcal{V}$ be a small domain containing $\left(0,0\right)$
and take $M,\widetilde{M}\in\mathcal{V}\cap\ssect$. Assume that $N$
comes from a resonant foliation $\fol{X_{R}}$ as in Lemma~\ref{lem:secto_norma}.
The relation $\hh_{*}\left(M\right)=\hh_{*}\left(\widetilde{M}\right)$
holds if and only if there exists a leaf $\mathcal{L}$ of the restricted
foliation $\fol{X_{R}}|_{\ssect}$ such that $M\in\mathcal{L}$ and
$\widetilde{M}\in\mathcal{L}$.
\item $\hh\left(\ssect\right)=\cc^{\times}$.
\item $\hh[][][0]\left(\ssect[][\pm]\right)$ is a punctured neighborhood
of $0^{\pm1}$. More precisely, there exists $A=A\left(k,\mu\right)>0$
and $c\left(k,\mu\right)>0$ such that if $c>c\left(k,\mu\right)$,
then:
\begin{align*}
\norm[\left(~_{N}^{j}\mathcal{H}^{\sharp}\right)^{\pm1}]{\ssect[][\pm]} & \leq A\ee^{-\nf ck+\norm[N]{\ssect}}.
\end{align*}
\end{enumerate}
\end{prop}

\subsection{Orbital necklaces and Martinet-Ramis orbital modulus}

\begin{figure}[H]
\hfill{}\includegraphics[width=0.5\columnwidth]{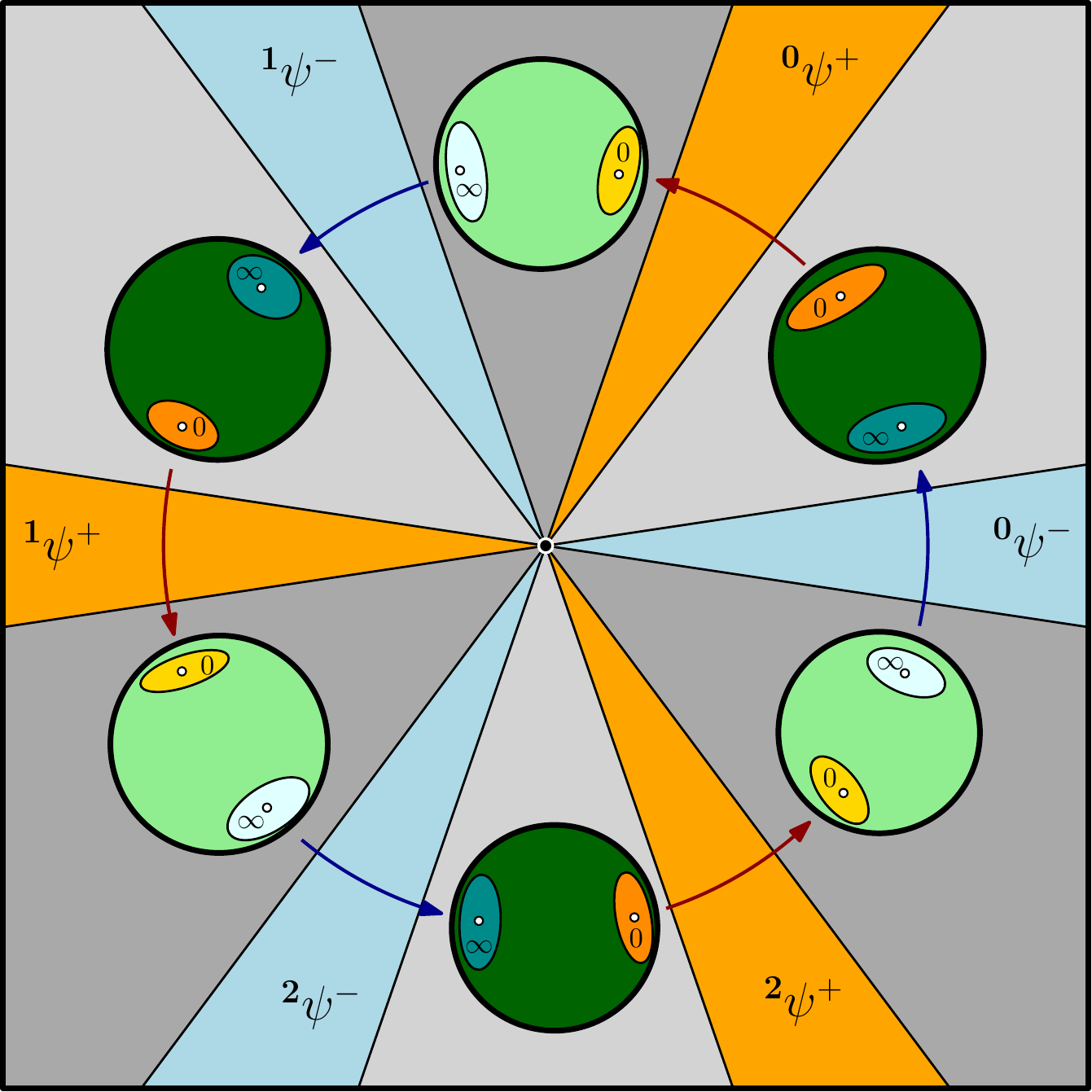}\hfill{}

\caption{An orbital necklace.}

\end{figure}

\begin{defn}
\label{def:MR_mod}~
\begin{enumerate}
\item An \textbf{orbital necklace} of order $k$ is the manifold (in general,
non-Hausdorff) obtained by gluing $2k$ Riemann spheres $\cbar$ near
their poles $0$ and $\infty$ by members of a collection of transition
maps $\left(\obj[][\pm]{\psi}\right)$, consisting in $2k$ germs
of a diffeomorphism near $0^{\pm1}$, one after the other in the circular
order on the indices given by Definition~\ref{def_sectors}.
\item We can choose a linear coordinate on each sphere in such a way that
each $\obj[][\pm]{\psi}$ is tangent to the identity but the last
one $\obj[k-1][+]{\psi}$, which in turn is tangent to a linear map
$h\mapsto\ee^{2\ii\pi\alpha}h$ for some $\alpha\in\quot{\cc}{\zz}$.
Let us call $\alpha$ the \textbf{residue} of the necklace.
\item The \textbf{Martinet-Ramis orbital modulus} of (a germ of) a $1:1$
resonant foliation $\fol{}$ is the orbital necklace defined by $\left(\obj[][\pm]{\psi}\right)$
obtained in the following way. Choose a neighborhood $\mathcal{V}$
of $\left(0,0\right)$ on which $\fol{}$ admits a holomorphic representative.
Pick a point $M\in\ssect[][\pm]$; depending on the considered intersection
$\ssect[][+]$ or $\ssect[][-]$, and according to Proposition~\ref{prop:secto_first-int},
there exist unique corresponding values $h^{\zi},h^{\iz}\in\cc$ of
the respective sectorial first-integrals. Set
\begin{align}
\obj[][-]{\psi}\left(h^{\iz}\right):=h^{\zi}\text{ or } & \obj[][+]{\psi}\left(h^{\zi}\right):=h^{\iz}.\label{eq:gluing_MR_mod}
\end{align}
\end{enumerate}
\end{defn}

\begin{rem}
~
\begin{enumerate}
\item By making another choice of the spherical coordinates we may assume
that this residue is distributed evenly between all transition mappings
$\obj{\psi}$, as explained below for Martinet-Ramis modulus.
\item Following Proposition~\ref{prop:secto_first-int}~1., the mapping
$h\mapsto\obj[][\pm]{\psi}\left(h\right)$ coming form a resonant
foliation is injective on the domain $\hh\left(\ssect[][\pm]\right)$.
Moreover, the choice of the leaf coordinate $\hh[][][0]$ (and of
the nonzero number $\sigma$) made in Definition~\ref{def:formal_model_first-integral}
implies
\begin{align}
\obj[][+]{\psi}\left(h\right) & =h\left(\sigma+\OO h\right)\label{eq:orbital_necklace_identifications}\\
\obj[][-]{\psi}\left(h\right) & =h\left(\sigma+\OO{\nf 1h}\right),\nonumber 
\end{align}
 so that $\mu$ is the residue of the corresponding orbital necklace
and it is distributed evenly.
\end{enumerate}
\end{rem}

By construction if $\mathcal{O}$ is an orbital necklace, then $\diff[\mathcal{O}]\simeq\GL 1{\cc}\times\zsk\times\zsk[2]$.
Indeed, once the obvious action of $\zsk\times\zsk[2]$ by translation
of the indices $j\mapsto j+\theta$ and exchange of polarity $+\leftrightarrow-$
is factored out, an element $\psi\in\diff[\mathcal{O}]$ must define
an automorphism of each sphere fixing both poles. Each such mapping
takes the form $h\mapsto ah$, and the nonzero constant $a$ must
be the same on each sphere for the global conformal structure to be
preserved. We are now ready to state the following fundamental classification
theorem upon which the present work is based.
\begin{namedthm}[Martinet-Ramis Theorem~\cite{MaRa-Res}]
The space of all equivalence classes (up to local analytic changes
of coordinates) of germs of a $1:1$ resonant saddle foliation belonging
to a given formal class $\left(k,\mu\right)$, is isomorphic to the
space of orbital necklaces of order $k$ and residue $\mu$ up to
the action of $\GL 1{\cc}\times\zsk\times\zsk[2]$.
\end{namedthm}

\section{\label{sec:Orbital_realization}Building a resonant saddle with prescribed
orbital necklace}

With hindsight, the heuristic for building the orbital modulus as
an orbital necklace seems rather natural: the analytic class of a
foliation is determined by the conformal structure of its leaf space.
What is less intuitive to grasp is how to solve the inverse problem
(or synthesis problem, or realization problem) for foliations.
\begin{problem*}
Being given an (abstract) orbital necklace, to prove that it comes
from a $p:q$ resonant foliation.
\end{problem*}
What J.~\noun{Martinet} and J.-P.~\noun{Ramis} did to solve the
orbital inverse problem was to use a powerful geometric black-box,
Newlander-Niremberg theorem, at the expense of loosing their grip
on an explicit realization. We propose here a process which is rather
explicit and use simple analytic ingredients. Before delving into
the details of the construction, let us first outline how it is performed
by following the strategy used in~\cite{SchaTey,RT2} and how the
twist parameter intervenes as in~\cite{TeySphere}. Here again we
perform the construction for $1:1$ resonant saddles, the minor differences
with the general $p:q$ case are being explained in Section~\ref{subsec:Realization_p:q}.

\bigskip{}

First, we express the transition maps provided by an orbital necklace
$\left(\obj[][\pm]{\psi}\right)$ as a collection of logarithmic data
$\left(\obj[][\pm]{\varphi}\right)$ consisting in $2k$ germs of
a holomorphic function near $0^{\pm1}$ such that: 
\begin{eqnarray}
\obj[][\pm]{\psi}\left(h\right) & = & \sigma h\exp\obj[][\pm]{\varphi}\left(h\right)\,\,\,\,,\,\obj[][\pm]{\varphi}\left(0^{\pm1}\right)=0.\label{eq:orbit_invar_expo}
\end{eqnarray}
Then the gluing conditions~(\ref{eq:gluing_MR_mod}), that are written
over $\ssect[][\pm]$ as:
\begin{align*}
\hh[][\zi]\left(u,y\right) & =\sigma\times\hh[][\iz]\left(u,y\right)\,\,\,\,\,\text{on }\ssect[][-],\\
\hh[j+1][\iz]\left(u,y\right) & =\sigma\times\hh[][\zi]\left(u,y\right)\,\,\,\,\,\text{on }\ssect[][+],
\end{align*}
amounts to solving the nonlinear Cousin problem
\begin{align*}
\begin{cases}
\obj[][\zi]N-\obj[][\iz]N & =\obj[][-]{\varphi}\circ\hh[][\iz]\\
\obj[j+1][\iz]N-\obj[][\zi]N & =\obj[][+]{\varphi}\circ\hh[][\zi]
\end{cases} & \tag{\ensuremath{\star}}
\end{align*}
where $N:=\left(\obj N\right)\in\mathcal{A}$ (Definition~\ref{def_first-integral})
is an unknown collection of sectorial functions. By solving~$\eqtag$
we realize the abstract necklace dynamics as transitions between canonical
first-integrals, yet still abstract at this stage. The fact that these
sectorial first-integrals do come from a concrete holomorphic foliation
is guaranteed by the following lemma.
\begin{lem}
\label{lem:vector_field_reconstruction}Let a logarithmic data $\varphi$
of an orbital necklace be given. Then there exists $c_{O}=c_{O}\left(k,\mu,\varphi\right)>0$
such that for all $c>c_{O}$ the following assertions hold. Assume
that $N$ solves~$\eqtag$ on the sectorial decomposition $\mathcal{V}$
as in Definition~\ref{def_sectors}. 
\begin{enumerate}
\item Each vector field 
\begin{eqnarray*}
\obj X & := & X_{0}-\frac{X_{0}\cdot\obj N_{*}}{1+Y\cdot\obj N_{*}}Y
\end{eqnarray*}
is holomorphic on $\ssect_{*}$ and admits $\hh_{*}$ for first-integral. 
\item The collection $\left(\obj X\right)$ is the restriction to sectors
of a $1:1$ resonant vector field $X$ holomorphic on the domain $\mathcal{V}_{*}$.
\end{enumerate}
\end{lem}

Once this vector field $X$ is synthesized we recognize it is actually
in the required normal form by bounding its growth as $u\to\infty$
for fixed $y$, see Corollary~\ref{cor:orbital_saddle}.
\begin{proof}
Define 
\begin{eqnarray*}
\obj R & := & -\frac{X_{0}\cdot\obj N_{*}}{1+Y\cdot\obj N_{*}}.
\end{eqnarray*}
Taking $c$ big enough ensures that $\obj R$ is holomorphic on $\ssect_{*}$
by forcing $\norm[Y\cdot\obj N_{*}]{\ssect}<1$, so that $\obj X=X_{0}+\obj R\times Y$
is holomorphic on $\ssect_{*}$ too. The claim is discussed in Corollary~\ref{cor:orbital_saddle}~1.
\begin{enumerate}
\item This is the same proof as that of Lemma~\ref{lem:secto_norma}~2. 
\item Because of Riemann's removable singularity theorem, each $\obj R$
(which is locally bounded near points of $\left\{ xy=0\right\} $,
see Corollary~\ref{cor:orbital_saddle}~1) is the restriction of
a function $R$ holomorphic on $\mathcal{V}_{*}$ if and only if $\obj[][\zi]R=\obj[][\iz]R$
on $\ssect[][-]_{*}$ and $\obj[j+1][\iz]R=\obj[][\zi]R$ on $\ssect[][+]_{*}$.
On the one hand, since $\hh[][\iz]_{*}$ is a first-integral of $\obj[][\iz]X$,
we have: 
\begin{eqnarray*}
\obj[][\iz]X\cdot\hh[][\zi]_{*} & = & \sigma\obj[][\iz]X\cdot\left(\hh[][\iz]\exp\obj[][-]{\varphi}\circ\hh[][\iz]_{*}\right)=0,
\end{eqnarray*}
while on the other hand we compute directly (after taking logarithmic
derivatives):
\begin{eqnarray*}
\obj[][\iz]X\cdot\hh[][\zi]_{*} & = & \hh[][\zi]_{*}\times\left(X_{0}\cdot\obj[][\zi]N_{*}+\obj[][\iz]R\times\left(1+Y\cdot\obj[][\zi]N_{*}\right)\right),
\end{eqnarray*}
as indeed $X_{0}\cdot\hh[][][0]_{*}=0$ and $Y\cdot\hh[][][0]_{*}=\hh[][][0]_{*}$.
Both identities considered together imply that $X_{0}\cdot\obj[][\zi]N_{*}+\obj[][\iz]R\times\left(1+Y\cdot\obj[][\zi]N_{*}\right)=0$,
which can be rewritten as $\obj[][\zi]R=\obj[][\iz]R$ on $\ssect[][-]_{*}$.
The proof in the other intersection is similar.
\end{enumerate}
\end{proof}
Solving~$\eqtag$ requires a refinement of the Cauchy-Heine transform
in order to recover functions $\obj N$ whose pairwise difference
in consecutive sectors is precisely $\obj[][\pm]{\varphi}$. Yet this
statement is imprecise since the collection $N$ itself determines
the first-integral $\hh$ that must be used to evaluate the right-hand
side $\obj[][\pm]{\varphi}\circ\hh$. The solution of the problem
must therefore be obtained as a fixed-point. 

In order to prove that this fixed-point method is well defined, and
to ascertain its convergence, we need to control the size of the neighborhoods
$\hh[][][N]\left(\ssect[][\pm]\right)$ of $0^{\pm1}$ so that they
fit within the disc of convergence of the corresponding $\obj{\varphi}$
and that everything takes place in a Banach space. This control is
gained through the twist parameter $c>0$, as follows from the estimates
of Lemma~\ref{lem:iteration} and Proposition~\ref{prop:contraction}.
This fact can already be surmised from the bounds of Proposition\ref{prop:secto_first-int}~3.

The rest of the section regards giving precise proofs and statements
leading to the resolution of~$\eqtag$, and Section~\ref{subsec:Orbital_realization}
concludes this section by providing a proof of the orbital part of
the Main~Theorem. 

\subsection{\label{subsec:Cauchy-Heine}Cauchy-Heine transform}
\begin{defn}
\label{def_adapted_sectors}We name $\mathcal{V}$ the collection
of sectors $\left(\ssect\right)$ in the variables $\left(u,y\right)$
as in Definition~\ref{def_sectors}. Let $c>c\left(k,\mu\right)$
be given as in Lemma~\ref{lem:secto_first-int_NF}.
\begin{enumerate}
\item A collection $\Delta=\left(\obj[][\pm]{\Delta}\right)$ of $2k$ star-shaped
domains of $\proj$ will be called \textbf{admissible} if:
\begin{itemize}
\item $\obj[][\xp]{\Delta}$ is a star-shaped domain centered at $0$,
\item $\obj[][\xm]{\Delta}$ is a star-shaped domain centered at $\infty$.
\end{itemize}
\item We say that $N=\left(\obj N\right)\in\mathcal{A}$ (Definition~\ref{def_first-integral})
is \textbf{adapted} to an admissible collection $\Delta$ if $\obj[][\pm]{\Delta}$
contains a disk centered at $0^{\pm1}$ of radius at least $\frac{2A}{\ee^{\nf ck}}\exp\norm[N]{\mathcal{V}}$.
\end{enumerate}
\end{defn}

Here is the cornerstone of the construction.
\begin{thm}
\label{thm:Cauchy-Heine}There exists a constant $c_{0}=c_{0}\left(k,\mu\right)\geq c\left(k,\mu\right)$
such that the upcoming statements hold for every fixed $c>c_{0}$.
Let $\Delta$ be an admissible collection as well as some $N\in\mathcal{A}$
adapted to $\Delta$. Take any collection $f=\left(\obj[][\pm]f\right)\in\holb[\Delta]'$.
There exists a constant $K=K\left(k,\mu\right)>0$, as well as a unique
collection $\Sigma\left(N,f\right)=\left(\obj{\Sigma}\right)\in\mathcal{A}$
such that the following properties hold. 
\begin{enumerate}
\item For all $j\in\zsk$ we have 
\begin{eqnarray*}
\obj[j+1][\iz]{\Sigma}-\obj[][\zi]{\Sigma} & = & \obj[][+]f\circ\hh[][\zi]\,\,\,\,\,\,\mbox{on }\ssect[][+],\\
\obj[j][\zi]{\Sigma}-\obj[][\iz]{\Sigma} & = & \obj[][-]f\circ\hh[][\iz]\,\,\,\,\,\,\mbox{on }\ssect[][-].
\end{eqnarray*}
\item $f\mapsto\Sigma\left(N,f\right)$ is a linear continuous map with
\begin{eqnarray*}
\norm[\Sigma\left(N,f\right)]{\mathcal{V}} & \leq & \frac{K}{c^{2}}\norm[f]{\Delta}'\ee^{\norm[N]{\mathcal{V}}};
\end{eqnarray*}
we recall that $\mathcal{V}$ is the collection of sectors $\left(\ssect\right)$.
\item $\obj{\Sigma}\left(u,y\right)=\OO{\sqrt{u}}$ as $u\to0$.
\item Moreover, with obvious notations:
\begin{eqnarray*}
\norm[u\frac{\partial\Sigma}{\partial u}]{\mathcal{V}} & \leq & \frac{K}{c^{2}}\norm[f']{\Delta}\ee^{\norm[N]{\mathcal{V}}}\left(1+\norm[u\frac{\partial N}{\partial u}]{\mathcal{V}}\right),\\
\norm[y\frac{\partial\Sigma}{\partial y}]{\mathcal{V}} & \leq & \frac{K}{c^{2}}\norm[f']{\Delta}\ee^{\norm[N]{\mathcal{V}}}\left(1+\norm[y\frac{\partial N}{\partial y}]{\mathcal{V}}\right),
\end{eqnarray*}
\end{enumerate}
\end{thm}

\begin{proof}
~
\begin{enumerate}
\item and~2. Here the $y$-variable plays the role of a parameter and is
supposed to be fixed. The functions $\left(\obj{\Sigma}\right)$ are
built by integrating $\obj[][\pm]f\circ\hh$ against some kernel we
describe below and along half-lines $\obj[][\pm]{\Gamma_{\pm}}$ bounding
the sectors $\sect[][\pm]$ that are provided with the orientation
$0\to\infty$, as in Figure.~\ref{fig:C-H_contour}. For the sake
of simplicity we only deal with the case $k=1$ (and drop the index
$j$ altogether), the general case resulting from an immediate adaptation
of what has been done in~\cite[Theorem 2.5]{SchaTey}.

Being given $f=\left(f^{+},f^{-}\right)$ meeting the hypothesis,
we define 
\begin{align}
\Sigma^{\zi}\left(u,y\right) & :=\frac{\sqrt{u}}{2\ii\pi}\int_{\Gamma_{-}^{-}}\frac{f^{-}\circ\hh[~][\iz]\left(z,y\right)}{\sqrt{z}\left(z-u\right)}\dd z\label{eq:CH_definition}\\
 & -\frac{\sqrt{u}}{2\ii\pi}\int_{\Gamma_{-}^{+}}\frac{f^{+}\circ\hh[~][\zi]\left(z,y\right)}{\sqrt{z}\left(z-u\right)}\dd z\nonumber \\
\Sigma^{\iz}\left(u,y\right) & :=\frac{\sqrt{u}}{2\ii\pi}\int_{\Gamma_{+}^{-}}\frac{f^{-}\circ\hh[~][\iz]\left(z,y\right)}{\sqrt{z}\left(z-u\right)}\dd z\\
 & -\frac{\sqrt{u}}{2\ii\pi}\int_{\Gamma_{+}^{+}}\frac{f^{+}\circ\hh[~][\zi]\left(z,y\right)}{\sqrt{z}\left(z-u\right)}\dd z.\nonumber 
\end{align}
Clearly these integrals are:
\begin{itemize}
\item well defined since $N$ is adapted to $\Delta$ and $\left|\hh[][][N]\left(z,y\right)\right|^{\pm1}\leq\frac{A}{\ee^{c}}\exp\norm[N]{}$
if $z\in\sect[~][\pm]$;
\item absolutely convergent because of the flatness of the exponential term
coming from $f^{\pm}\circ\hh[~]$, since $f^{\pm}\left(h\right)=\OO h$.
\end{itemize}
Properties~1. and~2. have been established in~\cite[Proposition 4.11]{TeySphere}.
The main point of the argument is the following: if $\left(u,y\right)\in\ssect[~][+]$,
then the Cauchy formula yields 
\begin{align*}
\Sigma^{\zi}\left(u,y\right)-\Sigma^{\iz}\left(u,y\right)= & \frac{\sqrt{u}}{2\ii\pi}\oint_{\Gamma_{-}^{-}-\Gamma_{+}^{-}}\frac{f^{-}\circ\hh[~][\iz]\left(z,y\right)}{\sqrt{z}\left(z-u\right)}\dd z\\
 & +\frac{\sqrt{u}}{2\ii\pi}\oint_{\Gamma_{+}^{+}-\Gamma_{-}^{+}}\frac{f^{+}\circ\hh[~][\zi]\left(z,y\right)}{\sqrt{z}\left(z-u\right)}\dd z\\
= & 0+2\ii\pi\times\frac{\sqrt{u}}{2\ii\pi}\frac{f^{+}\circ\hh[~][\zi]\left(z,y\right)}{\sqrt{u}}\\
= & f^{+}\circ\hh[~][\zi]\left(u,y\right),
\end{align*}
with a similar identity involving $f^{-}$ when $u\in V^{-}$. Of
course one must apply the Cauchy formula on a compact contour and
take a limit, but the flatness of the integrand ensures the process
actually works out nicely. Analogous details are dealt with in~\cite[Theorem 2.5]{SchaTey}.\\

\item[3.] Because $\hh$ is $1$-flat at $0$ and $\infty$, the contribution
of $\frac{1}{\sqrt{z}}$ to the integral is irrelevant: we would have
shown that the integral without that term is bounded. Hence $\Sigma$
is $\OO{\sqrt{u}}$.\\
We deduce from this asymptotic bound the fact that $\Sigma$ is unique.
Indeed if $\widetilde{\Sigma}\in\mathcal{A}$ is another collection
satisfying 1. and $\widetilde{\Sigma}=\OO{\sqrt{u}}$, then for fixed
$y$ the sectorial functions $\obj C~:~u\mapsto\obj{\Sigma}\left(u,y\right)-\obj{\widetilde{\Sigma}}\left(u,y\right)$
coincide on consecutive intersections, hence are sectorial restrictions
of a holomorphic function $C$ on $\cc^{\times}$. Since $C$ is bounded
it extends to an entire function thanks to Riemann's theorem on removable
singularity, thus a constant according to Liouville's theorem. But
this constant vanishes because $\Sigma$ and $\widetilde{\Sigma}$
are $\OO{\sqrt{u}}$, therefore $\Sigma=\widetilde{\Sigma}$.
\item[4.] These estimates follow in exactly the same manner as their counterpart
in~\cite[Theorem 2.5]{SchaTey}.
\end{enumerate}
\end{proof}
\begin{figure}[H]
\hfill{}\includegraphics[width=0.5\columnwidth]{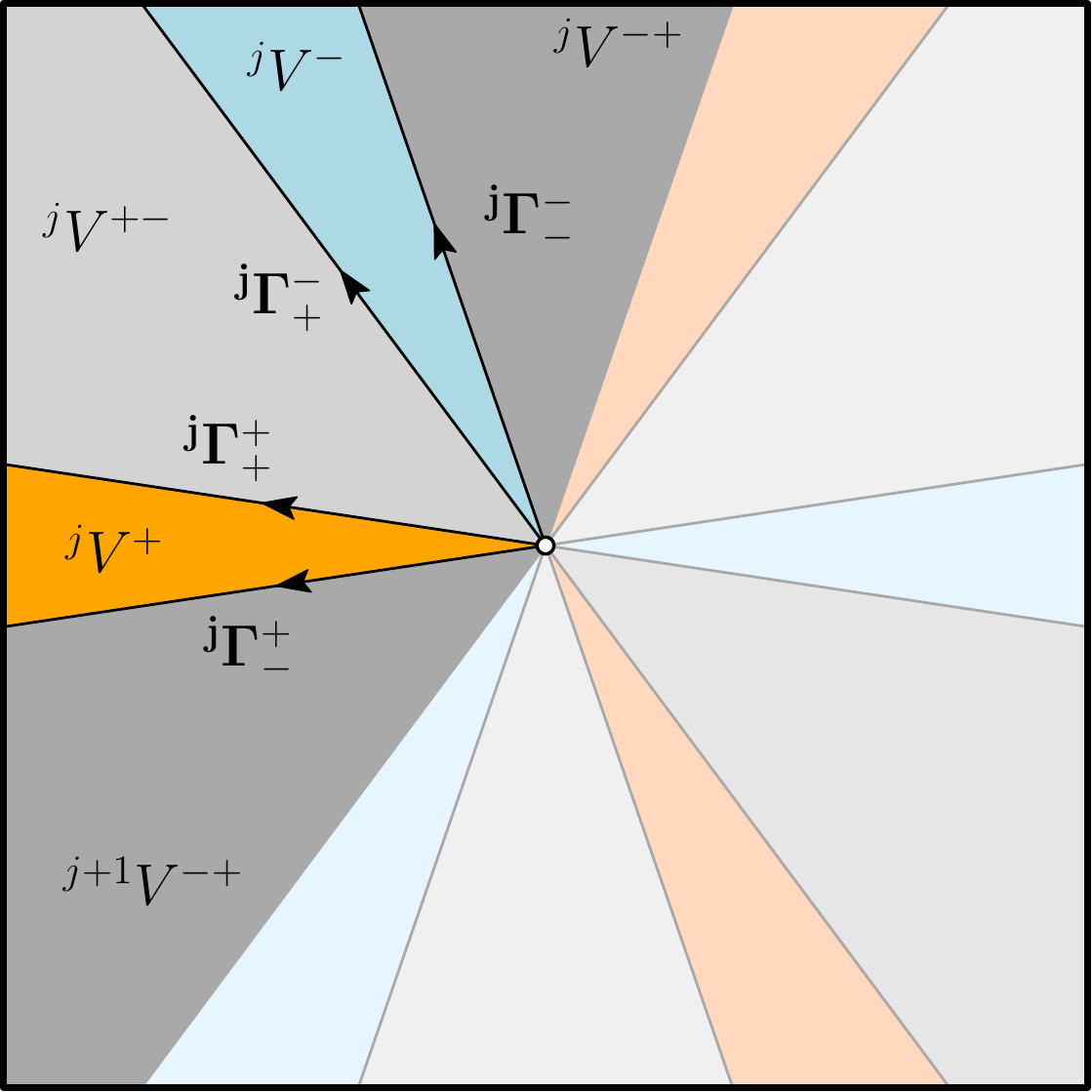}\hfill{}

\caption{\label{fig:C-H_contour}Contours used for the Cauchy-Heine transform.}
\end{figure}

\begin{rem}
\label{rem:inversion}The choice of the normalizing factor $\frac{\sqrt{u}}{\sqrt{z}}$
in the definition~(\ref{eq:CH_definition}) of $\obj{\Sigma}$ has
been done to ensure a ``nice'' behavior under the involution $\tau~:~u\mapsto\frac{1}{u}$.
This has been discussed in~\cite[Section 4.3, Proposition 4.11]{TeySphere},
and we get back to that fact in Section~\ref{sec:Uniqueness}.
\end{rem}

\subsection{\label{subsec:Fixed-point}The fixed-point method}

Let $\varphi:=\left(\obj{\varphi}\right)$ be the logarithmic data
associated to an orbital necklace and let $\Delta$ be an admissible
collection of open discs centered at $0^{\pm1}$ such that $\overline{\obj{\Delta}}$
lies in the domain of convergence of $\obj{\varphi}$. Consider the
partial map built from Theorem~\ref{thm:Cauchy-Heine}:
\begin{align*}
\Lambda:=\Sigma\left(\bullet,\varphi\right)~:~N & \longmapsto\Sigma\left(N,\varphi\right)~,
\end{align*}
which is defined on the space of all functions $N$ adapted to $\Delta$.
Starting from the collection $N_{0}:=\left(0\right)\in\mathcal{A}$,
which is of course adapted to $\Delta$ if $c$ is large enough, we
can compute $N_{1}:=\Lambda\left(N_{0}\right)\in\mathcal{A}$. We
wish to iterate that map to obtain a sequence $\left(N_{n}\right)_{n}$
lying in the Banach space $\mathcal{A}$.
\begin{lem}
\label{lem:iteration}Let $\mathcal{B}$ be the unit ball of $\mathcal{A}$.
There exists $c_{1}=c_{1}\left(k,\mu,\varphi\right)\geq c_{0}\left(k,\mu\right)$
such that for all $c>c_{1}$, every element of $\mathcal{B}$ is adapted
to $\Delta$ and $\Lambda$ is a self-map of $\mathcal{B}$.
\end{lem}

This lemma ensures that the sequence $\left(N_{n}\right)_{n}$ is
well defined and lies in $\mathcal{B}$.
\begin{proof}
Let $0<\rho$ be less than the minimum of all the radii of convergence
of $\obj{\varphi}$. We need to ensure that, for all $c$ and all
$N\in\mathcal{B}$: 
\begin{align*}
\frac{2A}{\ee^{\nf ck}}\exp\norm[N]{\mathcal{V}} & \leq\rho,
\end{align*}
where $\mathcal{V}$ is the collection of fibered sectors $\left(\ssect\right)$.
Assume that 
\begin{align*}
\norm[N]{\mathcal{V}} & \leq1,
\end{align*}
so that $N$ is adapted to $\Delta$ whenever 
\begin{align*}
c & >k\ln\left(\frac{2A}{\ee\rho}\right).
\end{align*}
Because of Theorem~\ref{thm:Cauchy-Heine}~2. one has
\begin{align*}
\norm[\Lambda\left(N\right)]{\mathcal{V}} & \leq\frac{K\ee}{c^{2}}\norm[\varphi]{\Delta}'.
\end{align*}
By taking $c>\sqrt{K\ee\norm[\varphi]{\Delta}'}$ we therefore ensure
that $\norm[\Lambda\left(N\right)]{\mathcal{V}}\leq1$, as expected.
Hence we may define $c_{1}:=\max\left\{ \sqrt{K\ee\norm[\varphi]{\Delta}'},k\ln\left(\frac{2A}{\ee\rho}\right)\right\} $.
\end{proof}
\begin{prop}
\label{prop:contraction}There exists $c_{2}=c_{2}\left(k,\mu,\varphi\right)\geq c_{1}$
such that, for all $c>c_{2}$ the mapping $\Lambda$ is a $\frac{1}{2}$-contracting
self-map of $\mathcal{B}$. In particular, the sequence $\left(N_{n}\right)_{n}$
converges towards the unique fixed point $N$ of $\Lambda|_{\mathcal{B}}$.
\end{prop}

\begin{proof}
The existence of this bound follows immediately from~\cite[Proposition 4.13]{TeySphere}.
Although the cited result is for the nonparametric version of Theorem~\ref{thm:Cauchy-Heine}
(\emph{i.e.} for $y:=1$), it gives the conclusion here too since
$\left|y\right|$ takes values bounded by $2$.
\end{proof}

\subsection{\label{subsec:Orbital_realization}Proof of the orbital realization}

Let us summarize what we have done so far. Starting from an orbital
necklace in a formal orbital class $\left(k,\mu\right)$, we consider
its logarithmic nonlinear part $\varphi=\left(\obj{\varphi}\right)$.
By taking $c>c_{2}$ as in Proposition~\ref{prop:contraction}, the
sequence $\left(\Lambda^{\circ n}\left(0\right)\right)_{n\geq0}$
converges towards a fixed-point $N$ of $\Lambda$, \emph{i.e.} a
solution in $\mathcal{A}$ of~$\left(\star\right)$. In order to
complete the orbital realization and recover $X_{R}$ in the expected
class, we wish to apply Lemma~\ref{lem:vector_field_reconstruction}
and next to prove that $R\in\mathcal{C}_{*}$, where
\begin{eqnarray*}
\mathcal{C}_{*} & := & \left\{ F~:~F\left(x,y\right)=y\sum_{n=1}^{2k}u_{*}^{n}f_{n}\left(y\right)~,~f_{n}\in\holb[\left\{ \left|y\right|<2\right\} ]\right\} .
\end{eqnarray*}
For this we need the following corollary, which settles the orbital
realization for $1:1$ resonant saddles.
\begin{cor}
\label{cor:orbital_saddle}There exists $c_{O}=c_{O}\left(k,\mu,\varphi\right)\geq c_{2}$
such that for all $c>c_{O}$ the following assertions hold.
\begin{enumerate}
\item Each function $\obj R:=-\frac{X_{0}\cdot\obj N_{*}}{1+Y\cdot\obj N_{*}}$
belongs to $\holf[\ssect_{*}]$. Moreover this function is locally
bounded near $\left\{ xy=0\right\} $.
\item These functions are actually the restrictions to sectors of a function
$R\in\mathcal{C}_{*}$.
\end{enumerate}
\end{cor}

\begin{proof}
First notice that by construction $\obj N\in\holb[{\sect\times D^{\sharp}}],$where
$D^{\zi}=\left\{ y~:~\left|y\right|<2\right\} $ and $D^{\iz}=\left\{ y~:~1<\left|y\right|<2\right\} $.
The intersection of the latter domains is the annulus 
\begin{align*}
C & :=\left\{ 1<\left|y\right|<2\right\} .
\end{align*}

For the sake of readability let us drop all indices. Observe that
\begin{align*}
x\ppp{N_{*}}x\left(x,y\right) & =u\ppp Nu\left(u\left(x,y\right),y\right)\\
y\ppp{N_{*}}y\left(x,y\right) & =u\ppp Nu\left(u\left(x,y\right),y\right)+\ppp Ny\left(u\left(x,y\right),y\right).
\end{align*}
In particular $Y\cdot N_{*}=\left(y\ppp Ny\right)_{*}$ and $X_{0}\cdot N_{*}=\left(u^{k+1}\ppp Nu+y\left(1+\mu u^{k}\right)\ppp Ny\right)_{*}$.
\begin{enumerate}
\item Since $\norm[N]{\mathcal{V}}\leq1$ (Lemma~\ref{lem:iteration}),
the estimates of Theorem~\ref{thm:Cauchy-Heine}~4. imply that 
\begin{eqnarray*}
\frac{\norm[y\ppp Ny]{\mathcal{V}}}{1+\norm[y\ppp Ny]{\mathcal{V}}} & \leq & \frac{K\ee}{c^{2}}\norm[\varphi']{\Delta}.
\end{eqnarray*}
If $c$ is large enough, then the right-hand side is smaller than
$\frac{1}{3}$ so that
\begin{align*}
\norm[y\ppp Ny]{} & \leq\frac{1}{2}.
\end{align*}
Therefore $Y\cdot N_{*}$ has a $\mathcal{V}_{*}$-norm less than
$\frac{1}{2}$ and $1+Y\cdot N_{*}$ cannot vanish on $\ssect_{*}$.
Moreover the latter modulus is bounded from below by $\frac{1}{2}$,
which gives the conclusion.
\item The fact that the sectorial functions $\obj R$ glue to some $R\in\holf[\mathcal{V}_{*}]$
has been explained in Lemma~\ref{lem:vector_field_reconstruction}.
The construction makes clear that $R=Q_{*}$ where
\begin{align*}
Q & :=\frac{u^{k+1}\ppp Nu+\left(c\left(1-u^{2k}\right)+\mu u^{k}\right)y\ppp Ny}{1+y\ppp Ny}\in\holf[\cc\times C].
\end{align*}
Let us bound the growth of $u\mapsto Q\left(u,y\right)$ as $u\to\infty$
with $y\in C$ fixed. We have:
\begin{align*}
\left|Q\left(u,y\right)\right| & \leq2\left|u\right|^{k}\left|u\ppp Nu\left(u,y\right)\right|+2\left(c\left|1-u^{2k}\right|+\mu\left|u\right|^{k}\right)\left|y\ppp Ny\left(u,y\right)\right|\\
 & \leq\left|u\right|^{k}\frac{2K\ee}{c^{2}}\norm[\varphi']{\Delta}+c\left(1+\left|u\right|^{2k}\right)+\left|\mu\right|\left|u^{k}\right|.
\end{align*}
As a conclusion $u\mapsto Q\left(u,y\right)$ is $\OO{u^{2k}}$, hence
$Q\in\holf[C]\left[u\right]_{\leq2k}$.\\
So far we have reached the point where $R\left(x,y\right)=\sum_{n=0}^{2k}f_{n}\left(y\right)\left(xy\right)^{n}$
for some functions $f_{n}\in\holb[C]$. But the construction actually
yields a function $R$ which is holomorphic on $\mathcal{V}_{*}$,
and this domain contains in addition to $\cc\times C$ a part corresponding
to $\left\{ \left(x,y\right)\in\mathcal{V}_{*}~:~\left|y\right|<2\text{ and }u\in\sect[][\zi]\right\} $.
For instance if we fix $x:=1$ then $R\left(1,y\right)$ is holomorphic
on the sector $V_{0}:=\left\{ 0<\left|y\right|<2,~\left|k\arg y-\pi\frac{2j+1}{2}\right|<\frac{5\pi}{8}\right\} $.
In particular each $f_{n}$ is holomorphic and bounded on $C\cup V_{0}$.
But as $x:=\ee^{\ii\theta}$ for $\theta\in\left[0,2\pi\right]$ moves
along the unit circle, we deduce that $y\mapsto R\left(\ee^{\ii\theta},y\right)$,
and thus each $f_{n}$, is holomorphic on every sector $V_{\theta}:=\left\{ 0<\left|y\right|<2,~\left|k\arg y+k\theta-\pi\frac{2j+1}{2}\right|<\frac{5\pi}{8}\right\} $,
whose union covers $\left\{ 0<\left|y\right|<2\right\} $. As a conclusion
$f_{n}$ is holomorphic and bounded on the whole punctured disc $\left\{ 0<\left|y\right|<2\right\} $,
therefore $f_{n}\in\holb[\left\{ \left|y\right|<2\right\} ]$.\\
Because $R=-X_{R}\cdot N_{*}$ (Lemma~\ref{lem:secto_norma}) and
since $\ppp Nu\left(u,y\right)=\OO y$ in the sector $\sect[0][+-]\times D^{+-}$
(see~\ref{eq:CH_definition}), we may take $y:=0$ for fixed $u$
to conclude $Q\left(u,0\right)=0$, which shows $f_{n}\left(0\right)=0$.\\
In a similar fashion we wish to evaluate $Q\left(0,y\right)=c\frac{y\ppp Ny}{1+y\ppp Ny}\left(0,y\right)$.
But from Theorem~\ref{thm:Cauchy-Heine}~3. we have $y\ppp Ny\left(u,y\right)=\OO{\sqrt{u}}$
as $u\to0$, which finally gives $Q\left(0,y\right)=0$. This completes
the proof.
\end{enumerate}
\end{proof}

\subsection{\label{subsec:Realization_p:q}The case of a $p:q$ resonant saddle}

The general case of a $p:q$ saddle foliation follows exactly the
same steps as for $1:1$ saddles, and most of the necessary results
hold \emph{verbatim}. The main modification is the necessity to use
a convenient version of the model first-integral, namely:
\begin{align*}
H\left(u,y\right) & =y^{\nf 1q}u^{-\nf{\mu}q}\exp\frac{c\left(u^{-k}+u^{k}\right)}{qk}\\
u_{*} & =x^{q}y^{p}.
\end{align*}
The fact that $H$ is multivalued in the $y$-variable is not an obstacle
since the pullback sectors $\ssect_{*}$ are simply connected. 

All subsequent arguments work in the same way.

\section{\label{sec:Period}Period operator and its natural section}

This section is devoted to the proof of the Cohomological Theorem.
From now on we work within the domain $\mathcal{U}_{*}:=\cc\times\left\{ y~:~\left|y\right|<2\right\} $
on which we consider a holomorphic vector field in normal form 
\begin{eqnarray*}
X_{R} & := & xu_{*}^{k}\pp x+\left(1+\mu u_{*}+R\right)Y
\end{eqnarray*}
build in the previous section (Corollary~\ref{cor:orbital_saddle}).
We denote by 
\begin{align*}
\obj{\mathcal{U}} & :=\left\{ \left(x,y\right)\in\mathcal{U}_{*}~:~u\left(x,y\right)\in\sect,~\left|y\right|<2\right\} 
\end{align*}
the corresponding sectors and their pairwise intersections $\obj[][\pm]{\mathcal{U}}$.
We also let $\fol R$ stand for the induced holomorphic foliation
on $\mathcal{U}_{*}$.

Being given a germ $G$ of a holomorphic function at $\left(0,0\right)$,
that is $G\in\germ{x,y}$, we outline how to solve 
\begin{eqnarray}
X_{R}\cdot F & = & G\label{eq:eq-h}
\end{eqnarray}
 in order to provide the sectorial normalization of Lemma~\ref{lem:secto_norma}
with a more geometrical flavor, since the functions $\obj N$ involved
in the sectorial normalization of $X_{R}$ are solutions of the equation
$X_{R}\cdot N=-R$.
\begin{thm}
\label{thm:cohomog}Consider the cohomological equation~(\ref{eq:eq-h})
with $G\in\germ{x,y}$.
\begin{enumerate}
\item There exists a formal solution $F\in\frml{x,y}$ if and only if the
Taylor expansion of $G$ at $\left(0,0\right)$ does not contain terms
$u^{n}$ for $n\in\left\{ 0,\ldots,k\right\} $. We write $\germ{x,y}_{>k}$
the space of all such germs.
\item There exists a neighborhood $\Omega$ of $\left(0,0\right)$ on which
$G$ is holomorphic and bounded, such that to each $\left(x,y\right)\in\Omega\cap\obj{\mathcal{U}}$
one can attach a path 
\begin{align*}
\obj{\gamma}\left(x,y\right)~ & :~[0,+\infty[\to\Omega\cap\obj{\mathcal{U}}
\end{align*}
 tangent to $X_{R}$ that starts at $\left(x,y\right)$ and accumulates
on $\left(0,0\right)$. 
\item The parametric integral $\obj F:=\int_{\obj{\gamma}}G\frac{\dd u}{u^{k+1}}$
is convergent and defines a holomorphic function $\obj F\in\holb[{\Omega\cap\obj{\mathcal{U}}}]$
solving~(\ref{eq:eq-h}) if and only if $G\in\germ{x,y}_{>k}$.
\end{enumerate}
\end{thm}

The proof of this result is given in Section~\ref{subsec:cohomog_eq}. 
\begin{rem}
We wish to underline that~3. is another instance of a phenomenon
observed in~\cite{Tey-EqH} for saddle-node vector fields: if $\eta$
is a meromorphic time-form of a vector field $X$ (that is $\eta\left(X\right)=1$),
then integrals of $G\eta$ along asymptotic paths converge if and
only if $X\cdot F=G$ admits a formal solution $F$. We do not know
if it is true when $X$ is a quasi-resonant saddle (irrational eigenratio)
or has a nonreduced singularity at $\left(0,0\right)$.
\end{rem}

The difference of two consecutive sectorial solutions $\obj F$ is
a first-integral of $X_{R}$ over the pairwise intersections of the
corresponding sectors, and the fact that they do not agree measure
how far they are from being the restriction of a holomorphic function
on $\Omega$. (Indeed if those bounded functions were to agree on
all intersections, then we would apply Riemann's theorem on removable
singularity.) Besides, each such pairwise difference factors through
$\hh_{*}$, as explained below.
\begin{lem}
\label{lem:ring_first-int}Any first-integral $\phi\in\holf[{\Omega\cap\obj[][\pm]{\mathcal{U}}}]$
of $X_{R}$ factors through $\hh[][\pm\mp]_{*}$: there exists $f\in\holf[{\obj[][\pm]{\Delta}}]$
such that
\begin{eqnarray*}
\phi & = & f\circ\hh[][\pm\mp]_{*},
\end{eqnarray*}
where $\obj[][\pm]{\Delta}:=\hh[][\pm\mp]_{*}\left(\Omega\cap\obj[][\pm]{\mathcal{U}}\right)$.
Moreover if $\phi$ is bounded, then $f$ also is.
\end{lem}

\begin{proof}
The function $\phi$ is constant on the leaves of $\fol R$ and therefore
defines a holomorphic function on its space of leaves. The conclusion
follows from the fact that $\hh[][\pm\mp]_{*}$ is a coordinate on
the corresponding sectorial leaf space, as guaranteed by Lemma~\ref{lem:secto_norma}
(maybe at the expense of reducing the size of $\Omega$).
\end{proof}
Therefore one can build a linear operator
\begin{eqnarray*}
\per\,:\,G\in\germ{x,y}_{>k} & \longmapsto & \left(\obj[][\pm]f\right)
\end{eqnarray*}
such that
\begin{eqnarray*}
\obj[][\zi]F-\obj[][\iz]F & = & \obj[][-]f\circ\hh[][\iz]_{*}\\
\obj[j+1][\iz]F-\obj[][\zi]F & = & \obj[][+]f\circ\hh[][\zi]_{*}\,.
\end{eqnarray*}

\begin{defn}
\label{def_period}The previous operator $\per$ is called the \textbf{period
operator} associated to $X_{R}$. 
\end{defn}

\begin{rem}
~
\begin{enumerate}
\item The value of the period $\obj[][\pm]{\per}\left(G\right)\left(h\right)$
is obtained by computing the integral $\int_{\obj[][\pm]{\gamma}\left(h\right)}\frac{G}{}\frac{\dd x}{x}$,
where $\obj[][\pm]{\gamma}\left(h\right)$ is an asymptotic cycle
obtained by the concatenation of the asymptotic tangent paths described
in Theorem~\ref{thm:cohomog}~2. passing through $\left(x,y\right)\in\obj[][\pm]{\mathcal{U}}$
and such that $h=\hh[][\pm\mp]\left(x,y\right)$.
\item As it has been already pointed out above, the formal solution of $X_{R}\cdot F=G$
converges if and only if $\per\left(G\right)=0$.
\item When $y:=0$ \emph{resp.} $x:=0$, the formal solution $F\left(x,0\right)$
\emph{resp.} $F\left(0,y\right)$ is simply given by integrating $-pcx\pp xF\left(x,0\right)=G\left(x,0\right)$
\emph{resp}. $qcy\pp yF\left(0,y\right)=G\left(0,y\right)$, hence
it is analytic on $\left\{ xy=0\right\} $. As a consequence $\obj[][\pm]f\left(0^{\pm1}\right)=0$.
\end{enumerate}
\end{rem}

\begin{cor}
\label{cor:invar_as_period}Let a formal class $\left(k,\mu,P\right)$
be given. The complete modulus $\left(\obj[][\pm]{\psi},\obj[][\pm]f\right)$
of $Z_{G,R}$ can be expressed as periods along $X_{R}$, namely:
\begin{align*}
\left(\obj[][\pm]{\varphi}\right) & =\per\left(-R\right)\\
\left(\obj[][\pm]f\right) & =\per\left(\frac{1}{G}-\frac{1}{P_{*}}\right)
\end{align*}
where $\obj[][\pm]{\psi}=\sigma\id\exp\obj{\varphi}$.
\end{cor}

\begin{proof}
This is a direct consequence of what has been explained previously.
\end{proof}
The technique used to carry out the realization of an orbital necklace
in Section~\ref{sec:Orbital_realization} actually allows us to show
that $\per$ admits an explicit section. This is done by reusing the
refined Cauchy-Heine transform presented in Theorem~\ref{thm:Cauchy-Heine}.
\begin{thm}
\label{thm:natural_section}Consider the admissible collection $\Delta$
defined by $\left(\hh[][\pm\mp][N]\left(\obj[][\pm]{\mathcal{U}}\right)\right)$
and let a collection $f=\left(\obj[][\pm]f\right)\in\holb[\Delta]'$
be given. Then there exists $\sec R\left(f\right)\in\mathcal{C}_{*}$
(the function space is defined in Section~\ref{subsec:Orbital_realization})
such that
\begin{eqnarray*}
\per[R]\left(\sec R\left(f\right)\right) & = & f.
\end{eqnarray*}
\end{thm}

The proof of this result is postponed until Section~\ref{subsec:section}.
\begin{rem}
In particular $\per~:~\mathcal{C}_{*}\longto\holb[\Delta]'$ is an
isomorphism of Banach spaces, with inverse its \textbf{natural section}
$\sec R$. The estimates given in Theorem~\ref{thm:Cauchy-Heine}
allow to give explicit bound on their norms.
\end{rem}

We finally are able to establish the temporal realization of the Main~Theorem.
\begin{cor}
Let $U$ be a holomorphic unit. Every vector field $UX_{R}$ is analytically
conjugate to a unique $Z_{G,R}$ with $G\in\mathcal{C}_{*}$.
\end{cor}

\begin{proof}
Let $U$ be given and consider the temporal normalizing equation $X_{R}\cdot\widehat{T}=\frac{1}{U}-\frac{1}{P_{*}}$,
which admits a formal solution if $P_{*}$ is given by the projection
of the Taylor series of $U$ on the space $\pol{u_{*}}_{\leq k}$.
According to the previous theorem, we can find $G\in\mathcal{C}_{*}$
such that $\per[R]\left(G\right)=\per[R]\left(\frac{1}{U}-\frac{1}{P_{*}}\right)$.
Then the cohomological equation $X_{R}\cdot T=\frac{1}{U}-\frac{1}{V}$
admits an analytic solution $T$, where $\frac{1}{V}:=\frac{1}{P_{*}}+G$,
and that implies $UX_{R}$ is analytically conjugate to $VX_{R}=Z_{G,R}$.
\end{proof}

\subsection{\label{subsec:cohomog_eq}Cohomological equations and their sectorial
solutions: proof of Theorem~\ref{thm:cohomog}}

Let us consider equation~(\ref{eq:eq-h}). 
\begin{enumerate}
\item Because $Y\cdot\left(x^{a}y^{b}\right)=\left(bq-ap\right)x^{a}y^{b}$
and $X_{R}=u^{k}x\pp x+\left(1+\mu u^{k}+R\right)Y$, no terms $u^{n}$
may belong to the image of $X_{R}$ if $n\leq k$. Conversely, the
coefficients of $F$ can be computed recursively by looking at larger
and larger homogeneous degrees $a+b$ of monomials $x^{a}y^{b}$ appearing
in the Taylor expansion of $G$. Details are left to the reader.
\item Let $\left(x_{*},y_{*}\right)$ be fixed in a sector $\obj{\mathcal{U}}$.
It is well know that the holonomy of $\fol{X_{R}}$, computed on $\left\{ y=y_{*}\right\} $
and obtained by lifting the path $\left(0,y\left(t\right)\right)=\left(0,y_{*}\ee^{\ii t}\right)$
into the leaves of the foliation, is a nonlinearizable parabolic germ
tangent to $x\mapsto\ee^{2\ii\pi\nf pq}x$. Starting close enough
to $0$, its forward or backwards orbit (the direction $t\to\pm\infty$
depending on the sector $\obj[][\zi]{\mathcal{U}}$ or $\obj[][\iz]{\mathcal{U}}$)
accumulates on $0$. The tangent path that is obtained this way can
be deformed within its supporting leaf $\mathcal{L}$ to land on $\left(0,0\right)$.
Indeed on a sufficiently small domain $\Omega\ni\left(0,0\right)$,
the leaf $\mathcal{L}\cap\Omega$ is very close to the set of level
$h:=y_{*}\exp\frac{cx_{*}^{-kq}y_{*}^{-kp}}{k}$ of $y\exp\frac{cu^{-k}}{k}$,
and that implicit relation can be inverted as
\begin{align*}
x\left(t\right) & =\frac{y\left(t\right)^{-\nf pq}}{\left(\frac{c}{k}\log\frac{h}{y\left(t\right)}\right)^{\nf 1{kq}}}.
\end{align*}
Each time $y\left(t\right)$ makes a turn around $0$, the amplitude
$\left|x\left(t\right)\right|$ is multiplied with a factor of order
about $\frac{1}{2\pi^{\nf 1{qk}}}$. By letting $\left|y\left(t\right)\right|$
goes to $0$ in a slower fashion than $y\left(t\right)$ winds around
$0$, we obtain a path that accumulates on $\left(0,0\right)$, \emph{e.g.}
by choosing $0<\alpha<\frac{1}{pk}$ and setting
\begin{align*}
y\left(t\right) & :=y_{*}\left(1+t\right)^{\alpha}\ee^{\ii t}.
\end{align*}
Hence
\begin{align*}
x\left(t\right) & \sim_{t\to\infty}\cst\times t^{-\alpha\nf pq-\nf 1{kq}}\ee^{-\ii t\nf pq}\\
u\left(t\right) & \sim_{t\to\infty}\cst\times t^{-\nf 1k}.
\end{align*}
One easily checks that the image of the lifted path $\gamma$ remains
in the given $u$-sector since
\begin{align*}
\arg u\left(t\right) & =-\frac{1}{kq}\arg\left(\cst+\alpha\ln\left(1+t\right)-\ii t\right)\longto_{t\to\infty}\pm\frac{\pi}{2kq}.
\end{align*}
\item According to the previous computations, we know that along the tangent
path $\gamma$ we have
\begin{align*}
\gamma^{*}\frac{\dd u}{u^{k+1}} & \sim_{t\to\infty}\cst\times\dd t.
\end{align*}
Integrating a monomial $x^{a}y^{b}$ along this path therefore yields
the estimate
\begin{align*}
\gamma^{*}\left(x^{a}y^{b}\frac{\dd u}{u^{k+1}}\right) & \sim_{t\to\infty}\cst\times t^{\frac{\alpha}{q}\left(ap-bq\right)-\frac{a}{kq}}\times\ee^{-\ii qt\left(ap-bq\right)}\dd t.
\end{align*}
There are two cases to consider.
\begin{itemize}
\item Either $ap=bq$, that is $x^{a}y^{b}=u^{n}$, and $\gamma^{*}\left(x^{a}y^{b}\frac{\dd u}{u^{k+1}}\right)\sim_{\infty}t^{-\nf nk}$
which is integrable if and only if $n>k$. In that case the integral
is absolutely convergent.
\item Or $ap\neq bq$ and $\gamma^{*}\left(x^{a}y^{b}\frac{\dd u}{u^{k+1}}\right)$
is conditionally integrable by Dirichlet's test.
\end{itemize}
Therefore $\obj F\left(x_{*},y_{*}\right)$ is a convergent integral
if and only if $G\in\germ{x,y}_{>k}$. It is clearly locally analytic
in $\left(x_{*},y_{*}\right)$. The fact that it is holomorphic on
$\Omega\cap\obj{\mathcal{U}}$ comes from the fact that if $\Omega$
is small enough, then every leaf of $\fol{X_{R}}|_{\obj{\mathcal{U}}}$
is simply connected according to the incompressibility result of \emph{e.g}.~\cite[Proposition 3.1]{TeyIncomp}.\hfill{}$\square$
\end{enumerate}

\subsection{\label{subsec:section}Natural section of the period operator: proof
of Theorem~\ref{thm:natural_section}}

We use Theorem~\ref{thm:Cauchy-Heine} by taking $N$ as the sectorial
normalizations of $X_{R}$ and $f=\left(\obj f\right)\in\mathcal{A}$:
we obtain sectorial functions $\left(\obj F\right):=\Sigma\left(N,f\right)$
solving the period Cousin problem. Define now
\begin{eqnarray*}
G & := & X_{R}\cdot\obj F\,.
\end{eqnarray*}
We conclude the proof by invoking the same arguments as in the proof
of Corollary~\ref{cor:orbital_saddle}~2. 
\begin{itemize}
\item Because by construction the difference of consecutive $\obj F$ is
a first integral of $X_{R}$, the function $G$ is independent on
the sector.
\item It is moreover locally bounded near $\left\{ xy=0\right\} $, so that
it extends to a function holomorphic on $\mathcal{V}_{*}$ by Riemann's
removable singularity theorem. 
\item Then for every fixed $y\in\left\{ 1<\left|y\right|<2\right\} $ we
have $G\left(x,y\right)=\OO{u^{2k}}$ as $u\to\infty$, \emph{i.e.}
$G\left(x,y\right)=\sum_{n=0}^{2k}u^{n}f_{n}\left(y\right)$.
\item Moreover $f_{0}\left(y\right)=0$ because $\obj F$ is $\OO{\sqrt{u}}$
as $u\to0$. 
\item Each function $f_{n}$ extends holomorphically to the disc $\left\{ \left|y\right|<2\right\} $
because of the shape of $\mathcal{V}_{*}$.
\end{itemize}
Therefore $G\in\mathcal{C}_{*}$ as expected.\hfill{}$\square$

\section{\label{sec:Uniqueness}Isotropy of resonant foliations and uniqueness
of the normal form}

As an application of the material introduced previously, we explain
how we deduce from the description of the cokernel of the period operator
that the vector fields obtained in the Normalization Theorem are essentially
unique in a given analytical class. Said differently, we wish to prove
that the automorphisms of the versal family $\left\{ Z_{G,R}~:~G,R\in\mathcal{C}_{*}\right\} $
are only given by the linear mappings $\left(x,y\right)\mapsto\left(\alpha x,\beta y\right)$
with $\left(\alpha^{q}\beta^{p}\right)^{k}=1$. We mainly rely on
the following structure theorem, which describes the isotropy group
of resonant foliations.
\begin{thm}
\label{thm:isotropy_factor}Take two vector fields $Z=U\left(X_{0}+RY\right)$
and $\widetilde{Z}=\widetilde{U}\left(X_{0}+\widetilde{R}Y\right)$,
not necessarily in normal form. If $\Psi$ is a conjugacy between
$Z$ and $\widetilde{Z}$, say $\Psi^{*}Z=\widetilde{Z}$, then it
can be factored as
\begin{align*}
\Psi & =\Lambda\circ\mathcal{N}\circ\mathcal{T}
\end{align*}
where:
\begin{itemize}
\item $\Lambda~:~\left(x,y\right)\longmapsto\left(\alpha x,\beta y\right)$
with $\left(\alpha^{q}\beta^{p}\right)^{k}=1$;
\item the diffeomorphism $\mathcal{N}\in\diff$ preserves the resonant monomial,
that is $\mathcal{N}=\flow YN{}=\left(x\ee^{-pN},y\ee^{qN}\right)$
for some germ $N$ at $\left(0,0\right)$ of a holomorphic function;
\item the diffeomorphism $\mathcal{T}\in\diff$ sends a leaf of $\fol{\widetilde{Z}}$
within itself (it is a tangential isotropy of $\fol{\widetilde{Z}}$),
that is $\mathcal{T}=\flow{\widetilde{Z}}T{}$ for some holomorphic
germ $T$.
\end{itemize}
\end{thm}

We prove this result in Section~\ref{subsec:Proof_of_Factor} below.
\begin{rem}
~
\begin{enumerate}
\item Even without applying the Normalization Theorem, it is known since
Dulac's works~\cite{Dulac,Dulac2} that any resonant saddle vector
field can be written $U\left(X_{0}+RY\right)$ in a convenient local
analytic chart about $\left(0,0\right)$. 
\item A similar result was obtained in~\cite{BerCerMez}, although the
foliations were presented in a different fashion and the fixed fibration
was $\left\{ x=\cst\right\} $ instead of the singular fibration $\left\{ u=\cst\right\} $
considered here.
\end{enumerate}
\end{rem}

In order to deduce the uniqueness statement of the Normalization Theorem
from Theorem~\ref{thm:isotropy_factor}, we only need to show that
$\mathcal{N}=\mathcal{T}=\id$ when $Z=Z_{G,R}$ and $\widetilde{Z}=Z_{\widetilde{G},\widetilde{R}}$
are in normal form. There exists $U$ a holomorphic unit such that
$X_{R}\circ\mathcal{N}=UX_{\widetilde{R}}\cdot\mathcal{N}$. Let us
write this relation in the basis $\left(u^{k}x\pp x,Y\right)$:
\begin{align*}
\begin{cases}
u^{k+1} & =U\times u^{k+1}\\
c\left(1-u^{2k}\right)+\mu u^{k}+R\left(u,y\ee^{qN}\right) & =U\times\left(c\left(1-u^{2k}\right)+\mu u^{k}+\widetilde{R}\left(u,y\right)+X_{\widetilde{R}}\cdot N\right)
\end{cases} & .
\end{align*}
We deduce from this system that $U=1$ on the one hand, while on the
other hand we have
\begin{align*}
X_{\widetilde{R}}\cdot N & =R\left(u,y\ee^{qN}\right)-y\widetilde{R}\left(u,y\right)=:L.
\end{align*}

\begin{prop}
\label{prop:inversion}$N\in y\germ y$.
\end{prop}

We postpone the proof of this result till Section~\ref{subsec:Proof_of_inversion};
in the meantime it tells us that $L\in\polg uy$. Since $N$ is holomorphic
we can assert that $\per\left(L\right)=0$ according to Cohomological
Theorem~2.(a). But the item~2.(b) of the same theorem implies that
$L=0$, which in turn implies $N=\cst$ and $R=\widetilde{R}$, as
expected.

To conclude the uniqueness of the normal form, we apply again the
previous argument: since $Z_{G,R}$ and $Z_{\widetilde{G},R}$ are
conjugate by $\mathcal{T}$, there exists a holomorphic germ $T$
such that $X_{R}\cdot T=G-\widetilde{G}\in\polg uy$, so that $G=\widetilde{G}$.
This concludes the proof of the Normalization Theorem.

\subsection{\label{subsec:Proof_of_Factor}Proof of Theorem~\ref{thm:isotropy_factor}}

Since $\Lambda$ can be read in the linear part of $\Psi$ (which
must be diagonal thanks to the form of $X_{0}$), we may as well assume
that $\Psi$ is tangent to the identity. Let $\left(\obj{\Psi}\right)$
\emph{resp}. $\left(\obj{\widetilde{\Psi}}\right)$ be the collection
of sectorial normalizations of $X:=X_{0}+RY$ \emph{resp}. $\widetilde{X}:=X_{0}+\widetilde{R}Y$
given in Lemma~\ref{thm:Cauchy-Heine}. By assumption $\Psi$ induces
the identity between the orbital necklaces of $\fol Z$ and $\fol{\widetilde{Z}}$,
hence $\left(\obj{\widetilde{\Psi}}\right)^{\circ-1}\circ\obj{\Psi}=:\mathcal{N}$
does not depend on the sector, and therefore extend as an element
$\mathcal{N}\in\diff$ such that $\mathcal{N}^{*}X=\widetilde{X}$.
But by construction $\mathcal{N}$ fixes the monomial $u$.

Consider now $\mathcal{T}:=\mathcal{N}^{\circ-1}\circ\Psi$, which
is a conjugacy between $\left(U\circ\mathcal{N}\right)X_{\widetilde{R}}$
and $\widetilde{U}X_{\widetilde{R}}$, that sends a leaf of $\fol{\widetilde{X}}$
into itself. Following the arguments of~\cite{BerCerMez} we deduce
that $\mathcal{T}$ has the expected form.\hfill{}$\square$

\subsection{\label{subsec:Proof_of_inversion}Proof of Proposition~\ref{prop:inversion}}

The argument takes place in the variables $\left(u,y\right)$. Indeed,
as it has been remarked in Section~\ref{subsec:Summability}, the
foliations $\fol{X_{R}}$ and $\fol{X_{\widetilde{R}}}$ correspond
to foliations $\fol{}$ and $\widetilde{\fol{}}$ in $\left(u,y\right)$-space
given by the differential 1-forms $\omega_{R}$ and $\omega_{\widetilde{R}}$
defined in~(\ref{eq:NF_u-y}). Because $\mathcal{N}$ fixes the resonant
monomial, it induces an orbital conjugacy $\Phi~:~\left(u,y\right)\mapsto\left(u,y\phi\left(u,y\right)\right)$
between $\omega_{R}$ and $\omega_{\widetilde{R}}$. Using a construction
\emph{à la} Mattei-Moussu (path-lifting technique), $\Phi$ extends
holomorphically to a neighborhood of $\left\{ y=0\right\} \simeq\cc$.
The key point is to extend it to a neighborhood of $\proj$ by using
a suitable compactification.

The foliations $\fol{}$ and $\widetilde{\fol{}}$ can be extended
to $\proj\times D$, where $D$ is the disc $\left\{ \left|y\right|<2\right\} $.
In the chart $\left(z,y\right)=\left(\frac{1}{u},y\right)$, the $1$-form
$\omega_{R}$ is written
\begin{align*}
z^{-k-1}\dd y-y\left(c\left(1-z^{-2k}\right)+\mu z^{-k}+R\left(z^{-1},y\right)\right)z^{-2}\dd z & ,
\end{align*}
which becomes holomorphic after multiplication with $z^{2k+2}$:
\begin{align*}
\omega_{R^{*}}\left(z,y\right) & =z^{k+1}\dd y+y\left(c\left(1-z^{2k}\right)-\mu z^{k}+R^{*}\left(z,y\right)\right)\dd z
\end{align*}
where
\begin{align*}
R^{*}\left(z,y\right) & =-z^{2k}R\left(z^{-1},y\right)=-y\sum_{n=1}^{2k}f_{n}\left(y\right)z^{2k-n}.
\end{align*}
The singularity of $\fol{}$ and $\widetilde{\fol{}}$ at $\left(\infty,0\right)$
is again a saddle-node, with formal orbital class $\left(k,-\mu\right)$.
The diffeomorphism $\Phi$ induces a conjugacy between their weak
holonomies computed on a transversal $\left\{ u=\cst\right\} $ close
to $\left\{ \infty\right\} \times D$, but unlike for the case of
resonant saddles this does not automatically imply that $\fol{}$
and $\widetilde{\fol{}}$ are $\Phi$-conjugate on a full neighborhood
of $\left(\infty,0\right)$. Fortunately that fact is guaranteed by
the construction of the normal form, as already hinted at in Remark~\ref{rem:inversion}.
\begin{lem}
Consider the involution $\tau$ of $\zsk\times\left\{ +,-\right\} $
defined by $\tau\left(j,\pm\right):=\left(k-1-j,\mp\right)$, and
let it act on collections $\left(\obj[][\pm]f\right)_{j\in\zsk}$
in the natural way. Then we have
\begin{align*}
\per[R^{*}]\left(R^{*}\right) & =-\tau^{*}\per[R]\left(R\right).
\end{align*}
\end{lem}

\begin{proof}
The action of $\tau$ on the indices $\left(j,\pm\right)$ corresponds
to the action of $\tau~:~u\mapsto\frac{1}{u}$ on the sectors in the
$u$-variable, \emph{i.e.} the sector $\sect[][\pm\mp]$ is sent to
$\sect[k-1-j][\mp\pm]=:\tau^{*}\sect$ by the transform. For every
component $\Gamma$ of $\partial\sect$ we compute:
\begin{align*}
I_{\Gamma}\left(\frac{1}{u},y\right)=\frac{\sqrt{\frac{1}{u}}}{2\ii\pi}\int_{\Gamma}\frac{\obj[][\pm]f\circ\hh[][\pm\mp]\left(z,y\right)}{\sqrt{z}\left(z-\frac{1}{u}\right)}\dd z & =\frac{\sqrt{u}}{2\ii\pi}\int_{\Gamma}\frac{\obj[][\pm]f\circ\hh[][\pm\mp]\left(z,y\right)}{\sqrt{z}\left(uz-1\right)}\dd z
\end{align*}
and perform the change of variable $w:=\tau\left(z\right)$. This
change of variable reverses the orientation of the image half-line
$\tau^{*}\Gamma\subset\partial\tau^{*}\sect$, therefore:
\begin{align*}
I_{\Gamma}\left(\frac{1}{u},y\right) & =-\frac{\sqrt{u}}{2\ii\pi}\int_{-\tau^{*}\Gamma}\frac{\obj[][\pm]f\circ\hh[][\pm\mp]\left(\frac{1}{w},y\right)}{\sqrt{\frac{1}{w}}\left(\frac{u}{w}-1\right)}\frac{\dd w}{w^{2}}\\
 & =-\frac{\sqrt{u}}{2\ii\pi}\int_{\tau^{*}\Gamma}\frac{\obj[][\pm]f\circ\hh[][\pm\mp]\left(\frac{1}{w},y\right)}{\sqrt{w}\left(w-u\right)}\dd w.
\end{align*}
But: 
\begin{align*}
\hh[][\pm\mp]\left(\frac{1}{w},y\right) & =\widehat{\hh[][\pm\mp][0]}\left(w,y\right)\exp\obj[][\pm\mp]N\left(\frac{1}{w},y\right)=\widehat{\hh[k-1-j][\mp\pm][\tau^{*}N]}\left(w,y\right),
\end{align*}
where $\tau^{*}N$ is the collection defined by $\obj[k-1-j][\mp\pm]N\left(u,y\right):=\obj[j][\pm\mp]N\left(\frac{1}{u},y\right)$
and $\widehat{\hh[][][0]}$ is the model first-integral where $\mu$
is replaced by $-\mu$. Finally: 
\begin{align*}
I_{\Gamma}\left(\frac{1}{u},y\right) & =-\widehat{I_{\tau^{*}\Gamma}}\left(u,y\right)
\end{align*}
where in the right-hand side we have replaced $\mu$ by $-\mu$. The
conclusion follows from~\ref{eq:CH_definition} as the collections$N$
is obtained as a fixed-point of the Cauchy-Heine operator built from
a linear combination of terms $I_{\Gamma}$.
\end{proof}
According to the discussion performed in Section~\ref{subsec:Period_section},
more precisely the period presentation of the orbital modulus laid
out in Corollary~\ref{cor:invar_as_period}, we deduce that the following
chain of identities holds:
\begin{align*}
\per[\widetilde{R}^{*}]\left(\widetilde{R}^{*}\right) & =-\tau^{*}\per[\widetilde{R}]\left(\widetilde{R}\right)\\
 & =-\tau^{*}\per[R]\left(R\right)=\per[R^{*}]\left(R^{*}\right)
\end{align*}
since $\fol{}$ and $\widetilde{\fol{}}$ are locally conjugate near
$\left(0,0\right)$. Hence, $\fol{}$ and $\widetilde{\fol{}}$ have
same orbital modulus at $\left(\infty,0\right)$ and are conversely
locally conjugate near $\left(\infty,0\right)$. It is well known
that there exists a unique conjugacy which is fibered in the $u$-variable
(see \emph{e.g}.~\cite{MaRa-SN,Tey-SN}), therefore $\Phi$ extends
as a fibered diffeomorphism $\left(u,y\right)\mapsto\Phi\left(u,y\right)$
on the whole $\proj\times D$. Because $\Phi\left(u,y\right)=\left(u,y\sum_{n\geq0}\phi_{n}\left(u\right)y^{n}\right)$
this means that each function $\phi_{n}$ extends holomorphically
as an entire and bounded function of $u$, hence a constant. This
completes the proof of Proposition~\ref{prop:inversion}.\hfill{}$\square$

\bibliographystyle{alpha}
\bibliography{biblio}

\end{document}